\documentstyle[12pt]{article}
\newtheorem{theorem}{Theorem}%[section]
\newtheorem{lemma}[theorem]{Lemma}
\newtheorem{proposition}[theorem]{Proposition}

\newtheorem{corollary}[theorem]{Corollary}
\newtheorem{remar}[theorem]{Remark}
\newenvironment{proof}{Proof:\ \ \ }{\QED}
\newenvironment{remark}{\begin{remar}\rm}{\end{remar}}

\newcommand{\QED}{{\unskip\nobreak\hfil\penalty50%
\hskip1em\hbox{}\nobreak\hfil $\Box$%
\parfillskip=0pt \finalhyphendemerits=0 \par\medskip\noindent}}
\newcommand{\bfind}[1]{\index{#1}{\bf #1}}
\newcommand{\gloss}[1]{#1\glossary{\protect #1}}
\newcommand{\n}{\par\noindent}

\newcommand{\sn}{\par\smallskip\noindent}
\newcommand{\mn}{\par\medskip\noindent}
\newcommand{\bn}{\par\bigskip\noindent}
\newcommand{\pars}{\par\smallskip}
\newcommand{\parm}{\par\medskip}
\newcommand{\parb}{\par\bigskip}
\newcommand{\Pzero}{\wp}
\newcommand{\fvklit}[1]{[#1]}

\newcommand{\mbb}[1]{\underline{#1}}
\newcommand{\sep}{^{\rm sep}}
\newcommand{\chara}{\mbox{\rm char}\,}
\newcommand{\trdeg}{\mbox{\rm trdeg}\,}
\newcommand{\subsetuneq}{\mathrel{\raisebox{.8ex}{\footnotesize%
$\displaystyle\mathop{\subset}_{\not=}$}}}
\newcommand{\adresse}{\par\bigskip \small\rm
   Mathematical Sciences Group, %\par
   University of Saskatchewan, \par
   106 Wiggins Road, %\par
   Saskatoon, Saskatchewan, Canada S7N 5E6 \par
   email: fvk@math.usask.ca \ \ --- \ \ home page:
   http://math.usask.ca/$\,\tilde{ }\,$fvk/index.html}
\font\tenlv=msbm10 scaled 1200
\font\sevenlv=msbm7 scaled 1200
\font\fivelv=msbm5 scaled 1200
\def\lv #1{{\mathchoice{{\hbox{\tenlv #1}}}{{\hbox{\tenlv #1}}}
{{\hbox{\sevenlv #1}}}{{\hbox{\fivelv #1}}}}}

\newcommand{\N}{\lv N}
\newcommand{\Q}{\lv Q}

\newcommand{\Z}{\lv Z}

\begin{document}
\title{Places of algebraic function fields in arbitrary
characteristic\footnote{The author would like to thank Florian Pop,
Alexander Prestel and Niels Schwartz for very inspiring
discussions, and Roland Auer for a thorough reading of an earlier
version.\n
This research was partially supported by a Canadian NSERC grant.\n
AMS Subject Classification: 12J10, 03C60}}
\author{Franz-Viktor Kuhlmann}
\date{21.\ 8.\ 2003}
%\subjclass{12J10, 03C60}
\maketitle
\begin{abstract}\noindent
{\footnotesize\rm
We consider the Zariski space of all places of an algebraic function
field $F|K$ of arbitrary characteristic and investigate its structure by
means of its patch topology. We show that certain sets of places with
nice properties (e.g., prime divisors, places of maximal rank,
zero-dimensional discrete places) lie dense in this topology.
Further, we give several equivalent characterizations of fields that are
large, in the sense of F.~Pop's Annals paper {\it Embedding problems
over large fields}. We also study the question whether a field $K$ is
existentially closed in an extension field $L$ if $L$ admits a
$K$-rational place. In the appendix, we prove the fact that the Zariski
space with the Zariski topology is quasi-compact and that it is a
spectral space.}
\end{abstract}
%
%
%
%ÄÄÄÄÄÄÄÄÄÄÄÄÄÄÄÄÄÄÄÄÄÄÄÄÄÄÄÄÄÄÄÄÄÄÄÄÄÄÄÄÄÄÄÄÄÄÄÄÄÄÄÄÄÄÄÄÄÄÄÄÄÄÄÄÄÄÄÄÄÄÄ
%
\section{Introduction}
%
%
%ÄÄÄÄÄÄÄÄÄÄÄÄÄÄÄÄÄÄÄÄÄÄÄÄÄÄÄÄÄÄÄÄÄÄÄÄÄÄÄÄÄÄÄÄÄÄÄÄÄÄÄÄÄÄÄÄÄÄÄÄÄÄÄÄÄÄÄÄÄÄÄ
%
\subsection{The Zariski space}
In this paper, we consider algebraic function fields $F|K$ of
arbitrary characteristic. For any place $P$ on $F$, the valuation ring
of $P$ will be denoted by ${\cal O}_P\,$, and its maximal ideal by
${\cal M}_P\,$. By a \bfind{place of $F|K$} we mean a place $P$ of $F$
whose restriction to $K$ is the identity.

Following [Z--SA], we denote by \gloss{$S(F|K)$} the set of all
valuations (or places) of $F$ that are trivial on $K$. It is called the
\bfind{Zariski space} (or \bfind{Zariski--Riemann manifold}) {\bf of}
$F|K$. As in [Z--SA] we shall make no distinction between equivalent
valuations, nor between equivalent places; so we are in fact talking
about the set of all valuation rings of $F$ which contain $K$ (and
therefore, $S(F|K)$ is indeed a set and not a proper class). Also,
$S(F|K)$ can be viewed as the set of all places of $F|K$, as every place
which is trivial on $K$ (i.e., an isomorphism on $K$) is equivalent to a
place whose restriction to $K$ is the identity.

Let $\Pzero$ be a fixed place on $K$. The set of all places of $F$ which
extend $\Pzero$ will be denoted by \gloss{$S(F|K\,;\,\Pzero)$}. Hence,
$S(F|K)=S(F|K\,;\,\mbox{\rm id}_K)$. Every set $S(F|K\,;\,\Pzero)$
carries the \bfind{Zariski-topology}, for which the basic open sets are
the sets of the form
\begin{equation}                            \label{ZT}
\{P\in S(F|K\,;\,\Pzero)\mid a_1\,,\ldots,\,a_k\in {\cal O}_P\}\;,
\end{equation}
where $k\in\N\cup\{0\}$ and $a_1,\ldots,a_k\in F$.

With this topology, $S(F|K\,;\,\Pzero)$ is a spectral space (cf.\ [H]);
in particular, it is quasi-compact. Its associated \bfind{patch
topology} (or \bfind{constructible topology}) is the finer topology
whose basic open sets are the sets of the form
\begin{equation}                            \label{PT}
\{P\in S(F|K\,;\,\Pzero)\mid a_1\,,\ldots,\,a_k\in {\cal O}_P\,;\,
b_1\,,\ldots,\,b_\ell\in {\cal M}_P\}\;,
\end{equation}
where $k,\ell\in\N\cup\{0\}$ and $a_1,\ldots,a_k,b_1,\ldots,b_\ell \in
F$. With the patch topology, $S(F|K\,;\,\Pzero)$ is a totally disconnected
compact Hausdorff space. For the convenience of the reader, we include a
proof for the compactness in the appendix. We derive it from a more
general model theoretic framework. See [Z--SA] or [V] for the more
classical proofs of the quasi-compactness of $S(F|K)$ under the Zariski
topology, and for further details about $S(F|K)$. See also [Z1] and
[Z2] for further information and the application of $S(F|K)$ to
algebraic geometry.

\parm
If $P$ is any place on a field $L$, then we will denote by $LP$ its
residue field, by $v_P$ an associated valuation (unique up to
equivalence) and by $v_PL$ its value group. For every place $P\in
S(F|K\,;\,\Pzero)$, its residue field $FP$ contains $K\Pzero$; we set
$\dim P:= \trdeg FP|K\Pzero\,$. Further, its value group $v_PF$ contains
$v_{\Pzero} K\,$. We set $\mbox{\rm rr} \,P:= \dim_\Q\,\Q\otimes
(v_PF/v_{\Pzero} K)$; this is called the \bfind{rational rank} of
$v_PF/v_{\Pzero}K$. Then we have the well known \bfind{Abhyankar
inequality}:
\begin{equation}                            \label{Ai}%\label{wtdk}
\trdeg F|K \>\geq\>\dim P \,+\, \mbox{\rm rr}\,P \;.
\end{equation}
If equality holds, then we call $P$ an \bfind{Abhyankar place}. The
value group $v_PF$ of every Abhyankar place $P$ is finitely generated
modulo $v_{\Pzero}K$, and the residue field $FP$ is a finitely generated
extension of $K\Pzero$ (see Corollary~\ref{fingentb} below).

As soon as the transcendence degree of $F|K$ is bigger than 1, the
Zariski space $S(F|K)$ will contain ``bad places'' (cf.\ [K5]).
The residue field $FP$ is then not necessarily finitely generated over
$K\Pzero=K$ (even though $F$ is finitely generated over $K$). This can
be a serious obstruction in the search for a $K$-rational specialization
of $P$; cf.\ [J--R]. (We call a place $P$ of a field $L$
{\bf $K$-rational} if $K$ is a subfield of $L$, $P$ is trivial on $K$,
and $LP=K$.) Similarly, the value group $v_PF$ of $P$ may not be
finitely generated. For such bad places, local uniformization is much
more difficult than for Abhyankar places. In fact, we show in [K--K]
that {\it Abhyankar places $P$ of $F|K$ admit local uniformization in
arbitrary characteristic, provided that $FP|K$ is separable}. To the
other extreme, we also prove local uniformization for rational discrete
places ([K6]). (We call a place $P\in S(F|K)$ \bfind{discrete} if its
value group is isomorphic to $\Z$, and \bfind{rational} if $FP=K$.) For
most other places, we have not yet been able to obtain local
uniformization without taking finite extensions of the function field
into the bargain ([K6]).

Therefore, the question arises whether we can ``replace bad places $Q$
by good places $P$''. Certainly, in doing so we want to keep a certain
amount of information unaltered. For instance, we could fix finitely
many elements on which $Q$ is finite and require that also $P$ is
finite on them. This amounts to asking whether the ``good'' places lie
Zariski-dense in the Zariski space. But if we mean by a ``good place''
just a place with finitely generated value group and residue field
finitely generated over $K$, then the answer is trivial: the identity is
a suitable place, as it lies in every Zariski-open neighborhood. The
situation becomes non-trivial when we work with the patch topology
instead of the Zariski topology. In addition, we can even try to keep
more information on values or residues, e.g., rational (in)dependence of
values or algebraic (in)dependence of residues.

In [K--P], such problems are solved in the case of $\chara K=0$ by an
application of the Ax--Kochen--Ershov--Theorem. In the present paper, we
will prove similar results for arbitrary characteristic, using the model
theory of tame fields, which we introduced in [K1] ([K2], [K7]).

%
%ÄÄÄÄÄÄÄÄÄÄÄÄÄÄÄÄÄÄÄÄÄÄÄÄÄÄÄÄÄÄÄÄÄÄÄÄÄÄÄÄÄÄÄÄÄÄÄÄÄÄÄÄÄÄÄÄÄÄÄÄÄÄÄÄÄÄÄÄÄÄÄ
%
\subsection{Dense subsets of the Zariski space}
{\bf When we say ``dense'' we will always mean ``dense with respect to
the patch topology''.} Throughout this section, we let $F|K$ be a
function field of transcendence degree $n$, and $\Pzero$ a place on $K$.
We set $p=\chara K$ if this is positive, and $p=1$ otherwise.

Our key result is the following generalization of the Main Theorem of
\fvklit{K--P}. Take any ordered abelian group $\Gamma$ and $r\in\N$. If
the direct product $\Gamma\oplus\bigoplus_r\Z$ of $\Gamma$ with $r$
copies of $\Z$ is equipped with an arbitrary extension of the ordering
of $\Gamma$, then it will be called an $r$-extension of $\Gamma$.

An extension $(K_1,P_1)\subseteq (K_2,P_2)$ is called \bfind{immediate}
if the canonical embedding of $K_1P_1$ in $K_2P_2$ and the canonical
embedding of $v_{P_1}K_1$ in $v_{P_2}K_2$ are onto.

\begin{theorem}                                   \label{MKP}
Take a place $Q\in S(F|K\,;\,\Pzero)$ and
\[a_1,\ldots,a_m\in F\;.\]
Then there exists a place $P\in S(F|K\,;\,\Pzero)$ with value group
finitely generated over $v_{\Pzero}K$ and residue field finitely
generated over $K\Pzero$, such that
\[a_iP \>=\> a_iQ \mbox{ \ \ and \ \ } v_Pa_i \>=\> v_Qa_i
\mbox{ \ \ for \ \ } 1\leq i \leq m\;.\]
Moreover, if $r_1$ and $d_1$ are natural numbers satisfying
\[\dim Q\leq d_1\;,\qquad \mbox{\rm rr}\, Q
%\langle v_Q a_1,\ldots,v_Q a_m\rangle
\leq r_1\mbox{ \ \ and \ \ } 1\leq r_1+d_1\leq n\;,\]
then $P$ may be chosen to satisfy in addition:
\par\smallskip
(a)\ \ $\dim P = d_1$ and $FP$ is a subfield of the rational function
field in $d_1-\dim Q$ variables over the perfect hull of $FQ$,
\par\smallskip
(b)\ \ $\mbox{\rm rr}\, P = r_1$ and $v_P F$ is a subgroup of an
arbitrarily chosen $(r_1-\mbox{\rm rr}\,Q)$-extension of
the $p$-divisible hull of $v_Q F$.
\sn
The above remains true even for $d_1=0=r_1\,$, provided that
each finite extension of $(K,\Pzero)$ admits an immediate extension of
transcendence degree $n$.
\end{theorem}

The last condition mentioned in the theorem holds for instance for all
$(K,\Pzero)$ for which the completion is of transcendence degree at
least $n$. Note that the case of $d_1=0=r_1$ only appears when $\Pzero$
is nontrivial.

\pars
If $v_Q a_i\geq 0$ for $1\leq i\leq k$, and if $b_1,\ldots,b_{\ell}\in
F$ such that $v_Qb_j>0$ for $1\leq j\leq\ell$, then we can choose $P$
according to the theorem such that also $v_P a_i\geq 0$ for $1\leq i\leq
k$ and $v_Pb_j>0$ for $1\leq j\leq\ell$. That is, we can find a $P$ with
the required properties in every open neighborhood of $Q$ w.r.t.\ the
patch topology.

If we choose $r_1=n-d_1$ then $P$ will be an
Abhyankar place. Hence our theorem yields:
\begin{corollary}                           \label{cor2}
The set of all places with finitely generated value group modulo
$v_{\Pzero}K$ and with residue field finitely generated over $K\Pzero$
lies dense in $S(F|K\,;\,\Pzero)$. The same holds for its subset of all
Abhyankar places.
\end{corollary}

\parb
In certain cases we would like to obtain value groups of smaller
rational rank; e.g., we may want to get discrete places in the case
where $\Pzero$ is trivial. A modification in the proof of
Theorem~\ref{MKP} yields the following result:
\begin{theorem}                             \label{MKPZ}
Take a place $Q\in S(F|K\,;\,\Pzero)$ and $a_1,\ldots,
a_m\in F$. Choose $r_1$ and $d_1$ such that
\[\dim Q\leq d_1\leq n-1\mbox{ \ \ and \ \ } 1\leq r_1\leq n-d_1\;.\]
Then there is a place $P$ such that
\[a_iP \> = \>a_iQ \;\;\mbox{ for } \;\; 1\leq i \leq m\]
and
\par\smallskip
(a)\ \ $\dim P = d_1$ and $FP$ is a subfield, finitely generated over
$K\Pzero\,$, of a purely trans\-cendental extension of transcendence
degree $d_1-\dim Q$ over the perfect hull of $FQ$,
\par\smallskip
(b)\ \ $\mbox{\rm rr}\, P = r_1$ and $v_P F$ is a subgroup, finitely
generated over $v_{\Pzero}K$, of an arbitrarily chosen $r_1$-extension
of the divisible hull of $v_{\Pzero} K$.
\sn
The above remains true even for $r_1=0$, provided that each finite
extension of $(K,\Pzero)$ admits an immediate extension of transcendence
degree $n$.
\end{theorem}

We deduce two corollaries for the Zariski space $S(F|K)$; we leave it to
the reader to formulate analogous results for $S(F|K\,;\,\Pzero)$. A place
$P\in S(F|K)$ of dimension $\trdeg F|K\, -1$ is called a \bfind{prime
divisor of $F|K$} (one also says that $P$ has \bfind{codimension 1}).
Every prime divisor is an Abhyankar place, has value group isomorphic to
$\Z$ and a residue field which is finitely generated over $K$ (cf.\
Lemma~\ref{prelBour} below). From the above theorem, applied with
$d_1=n-1$, we obtain the following result:
\begin{corollary}
The prime divisors of $F|K$ lie dense in $S(F|K)$.
\end{corollary}

If on the other hand we choose $d_1=\dim Q$, then $FP$ will be contained
in a finite purely inseparable extension of $FQ$. If $\dim Q=0$, i.e.,
$FQ|K$ is algebraic, then it follows that $FP$ is a finite extension of
$K$. If in addition $K$ is perfect and $Q$ is rational, then also $P$ is
rational. With $r_1=1$, we obtain:
\begin{corollary}                           \label{discrat}
If $K$ is perfect, then the discrete rational places lie dense in
the space of all rational places of $F|K$.
\end{corollary}

A place $P\in S(F|K\,;\,\Pzero)$ is called \bfind{rational} if $FP=
K\Pzero\,$. Further, $P$ is called a \bfind{place of maximal rank} if
$v_{\Pzero}K$ is a convex subgroup of $v_PF$, which implies that the
ordering of $v_PF$ canonically induces an ordering on
$v_PF/v_{\Pzero}K$, and the rank $\mbox{rk}\, v_PF/v_{\Pzero}K$ of
$v_PF/v_{\Pzero}K\,$ with respect to the induced ordering (the number of
proper convex subgroups of $v_PF/v_{\Pzero}K$) is equal to the
transcendence degree of $F|K$. Since the rational rank is always
bigger than or equal to the rank, inequality (\ref{wtdgeq}) of
Corollary~\ref{fingentb} shows that $\mbox{rk}\,v_PF/v_{\Pzero}K\leq\trdeg
F|K$. Every place of maximal rank is a zero-dimensional Abhyankar place.
We take $r_1=\trdeg F|K$ and $\bigoplus_{r_1}\Z\oplus v_{\Pzero}K$ to be
lexicographically ordered (that is, $v_{\Pzero}K$ is a convex subgroup
of $\bigoplus_{r_1}\Z\oplus v_{\Pzero}K$ and $(\bigoplus_{r_1}\Z\oplus
v_{\Pzero}K)/ v_{\Pzero}K$ is of rank $r_1$). Then we obtain from
Theorem~\ref{MKPZ}:
\begin{corollary}                           \label{maxrat}
If $K\Pzero$ is perfect, then the rational places of maximal rank lie
dense in the subspace of all rational places in $S(F|K\,;\,\Pzero)$.
\end{corollary}

In order to decrease the dimension of places, we complement
Theorem~\ref{MKPZ} by the following theorem, which we will
prove in Section~\ref{sectpr3}:

\begin{theorem}                             \label{MTdropd}
Take a place $Q\in S(F|K\,;\,\Pzero)$ and $a_1,\ldots,
a_m\in F$. Assume that $\dim Q>0$. Choose $r_1$ and $d_1$ such that
\[\mbox{\rm rr} \,Q +1 \leq r_1\leq n \mbox{ \ \ and \ \ }
0\leq d_1\leq n-r_1\;.\]
Then there is a place $P$ such that
\begin{equation}                            \label{valeq}
v_Pa_i \> = \>v_Qa_i \;\;\mbox{ for } \;\; 1\leq i \leq m
\end{equation}
and
\par\smallskip
(a)\ \ $\dim P = d_1$ and $FP$ is a subfield, finitely generated over
$K\Pzero\,$, of a purely trans\-cendental extension of the algebraic
closure of $FQ$,
\par\smallskip
(b)\ \ $\mbox{\rm rr}\, P = r_1$ and $v_P F$ is a subgroup, finitely
generated over $v_{\Pzero}K$, of an arbitrarily chosen $(r_1-\mbox{\rm
rr}\,Q-1)$-extension of a group $\Gamma$ which admits $\Z$ as a convex
subgroup such that $\Gamma/\Z$ is isomorphic to a subgroup of the
$p$-divisible hull of $v_QF$ which is also finitely generated over
$v_{\Pzero}K$.
\end{theorem}

\n
Here, assertion (\ref{valeq}) means that there is an embedding
$\iota$ of the group $\sum_{i=1}^{m}\Z v_Q a_i$ in $v_P F$ such that
$v_P a_i=\iota v_Q a_i\,$.

\parm
Taking $\Pzero=\mbox{id}_K$ and $d_1=0$, we obtain:
\begin{corollary}
The zero-dimensional places with finitely generated value group and
residue field a finite extension of $K$ lie dense in $S(F|K)$.
\end{corollary}

Applying first Theorem~\ref{MTdropd} and then Theorem~\ref{MKPZ}, we
obtain:
\begin{theorem}                             \label{MTbetter}
Take $d_1,r_1\in\N$ such that $d_1\geq 0$, $r_1\geq 1$ and $d_1+r_1
\leq\trdeg F|K$. Then the places $P$ with
\pars
(a) residue field $FP$ a subfield of a purely transcendental extension
of the algebraic closure of $K\Pzero$, finitely generated of
transcendence degree $d_1$ over $K\Pzero$, and
\pars
(b) value group $v_PF$ a subgroup of some fixed $r_1$-extension of the
divisible hull of $v_{\Pzero}K$ such that $v_PF/v_{\Pzero}K$ is of
rational rank $r_1$ and finitely generated,
\sn
lie dense in $S(F|K\,;\,\Pzero)$.
\end{theorem}

Taking $\Pzero=\mbox{id}_K$, $d_1=0$ and $r_1=1$, we obtain:
\begin{corollary}
The discrete zero-dimensional places with residue field a finite
extension of $K$ lie dense in $S(F|K)$.
\end{corollary}

If we apply Theorem~\ref{MTbetter} with $d_1=0$ and $r_1=\trdeg F|K$,
where we take the $r_1$-extension $\bigoplus_{r_1}\Z\oplus v_{\Pzero}K$
to be lexicographically ordered, we obtain:
\begin{corollary}
The zero-dimensional places of maximal rank with residue field a finite
extension of $K\Pzero$ lie dense in $S(F|K\,;\,\Pzero)$.
\end{corollary}

\parb
Corollary~\ref{discrat} and Corollary~\ref{maxrat} state
that if $K$ is perfect, then the discrete rational places and the
rational places of maximal rank lie dense in the space of all rational
places. We do not know whether Corollary~\ref{discrat} and
Corollary~\ref{maxrat} hold without the assumption that $K$ be
perfect. We need this assumption to deduce these results from
Theorem~\ref{MKPZ}. But if there is a rational place which admits a
strong form of local uniformization, then we can show the same result
without this assumption. We will say that a place $P$ of the function
field $F|K$ admits \bfind{smooth local uniformization} if there is a
model of $F|K$ on which $P$ is centered at a smooth point; in addition,
we require that if finitely many elements $a_1,\ldots,a_m\in {\cal O}_P$
are given, then the model can be chosen in such a way that they are
included in the coordinate ring. If $K$ is perfect, this is equivalent
to the usual notion of local uniformization where ``simple point'' is
used instead of ``smooth point''. In [K6] and [K--K] we show:
\begin{theorem}                             \label{slu}
Every rational discrete and every rational Abhyankar place admits smooth
local uniformization.
\end{theorem}

The following density result will be proved in Section~\ref{sectdrplu}:
\begin{theorem}                             \label{densrplu}
The rational discrete places and the rational places of maximal rank lie
dense in the space of all rational places of $F|K$ which admit smooth
local uniformization.
\end{theorem}

%
%ÄÄÄÄÄÄÄÄÄÄÄÄÄÄÄÄÄÄÄÄÄÄÄÄÄÄÄÄÄÄÄÄÄÄÄÄÄÄÄÄÄÄÄÄÄÄÄÄÄÄÄÄÄÄÄÄÄÄÄÄÄÄÄÄÄÄÄÄÄÄÄ
%
\subsection{Large fields}                   \label{sectlf}
Following F.~Pop [POP1,2], a field $K$ is called a {\bf large field} if
it satisfies one of the following equivalent conditions:
\sn
{\bf (LF)} \ {\it For every smooth curve over $K$ the set of rational
points is infinite if it is non-empty.}
\sn
{\bf (LF$'$)} \ {\it In every smooth, integral variety over $K$ the set
of rational points is Zariski-dense if it is non-empty.}
\sn
{\bf (LF$''$)} \ {\it For every function field $F|K$ in one variable
the set of rational places is infinite if it is non-empty.}
\sn
For the equivalence of (LF) and (LF$'$), note that the set of all smooth
$K$-curves through a given smooth $K$-rational point of an integral
$K$-variety $X$ is Zariski-dense in $X$. If (LF) holds, then the set of
$K$-rational points of any such curve is Zariski-dense in the curve,
which implies that the set of $K$-rational points of $X$ is
Zariski-dense in $X$. The equivalence of (LF) and (LF$''$) follows from
two well-known facts: a) every function field in one variable is the
function field of a smooth curve (cf.\ [HA], Chap.~I, Theorem 6.9),
and b) every $K$-rational point of a smooth curve gives rise to a
$K$-rational place. The latter is a special case of a much more general
result which we will need later:
\begin{theorem}                             \label{exratpl}
Assume that the affine irreducible variety $V$ defined over $K$ has a
simple $K$-rational point. Then its function field admits a rational
place of maximal rank, centered at this point.
\end{theorem}
This follows from results in [A] (see appendix A of [J--R]).

\parm
A field $K$ is existentially closed in an extension field $L$ if for
every $m,n\in\N$ and every choice of polynomials $f_1,\ldots, f_n,g\in
K[X_1,\ldots,X_m]$, whenever $f_1,\ldots,f_n$ have a common zero in
$L^m$ which is not a zero of $g$, they also have a common zero in $K^m$
which is not a zero of $g$. In Section~\ref{sectlfp}, we shall prove:

\begin{theorem}                             \label{lfecKt}
The following conditions are equivalent:
\sn
1) \ $K$ is a large field,
\sn
2) \ $K$ is existentially closed in every function field $F$ in one
variable over $K$ which admits a $K$-rational place,
\sn
3) \ $K$ is existentially closed in the henselization $K(t)^h$ of the
rational function field $K(t)$ with respect to the $t$-adic valuation,
\sn
4) \ $K$ is existentially closed in the field $K((t))$ of formal Laurent
series,
\sn
5) \ $K$ is existentially closed in every extension field which admits a
discrete $K$-rational place.
\end{theorem}
\n
The canonical $t$-adic place of the fields $K(t)^h$ and $K((t))$ is
discrete, and it is trivial on $K$ and $K$-rational. Therefore, 5)
implies 3) and 4).

\pars
In [L], Serge Lang proved that every field $K$ complete under a rank one
valuation is large. But this already follows from the fact that such a
field is henselian. Indeed, if a field $K$ admits a non-trivial
henselian valuation, then the
Implicit Function Theorem holds in~$K$ (cf.\ [P--Z]). Using this fact,
it is easy to show that $K$ satisfies (LF). On the other hand, it is
also easy to prove that $K$ satisfies condition 3) of the foregoing
theorem, and we will give the proof in Section~\ref{sectlfp} to
demonstrate the arguments that are typical for this model theoretic
approach. We note:
\begin{proposition}                         \label{hvilarge}
If a field $K$ admits a non-trivial henselian valuation, then it is
large.
\end{proposition}

\pars
In view of condition 5) of the above theorem, the question arises
whether the existence of a $K$-rational place of an extension field $L$
of a large field $K$ always implies that $K$ is existentially closed in
$L$. We will discuss this in the next section.

%
%ÄÄÄÄÄÄÄÄÄÄÄÄÄÄÄÄÄÄÄÄÄÄÄÄÄÄÄÄÄÄÄÄÄÄÄÄÄÄÄÄÄÄÄÄÄÄÄÄÄÄÄÄÄÄÄÄÄÄÄÄÄÄÄÄÄÄÄÄÄÄÄ
%
\subsection{Rational place $=$ existentially closed?} \label{sectrp=ec}
Condition 4) of Theorem~\ref{lfecKt} leads us to ask whether large
fields satisfy conditions which may appear to be even stronger. In fact,
we shall prove in Section~\ref{sectlfp}:
\begin{theorem}                           \label{pcond}
Let $K$ be a perfect field. Then the following conditions are equivalent:
\sn
1) \ $K$ is a large field,
\sn
2) \ $K$ is existentially closed in every power series field $K((G))$.
\sn
3) \ $K$ is existentially closed in every extension field $L$ which
admits a $K$-rational place.
\end{theorem}

In particular, we obtain:
\begin{theorem}                             \label{kecF}
Let $K$ be a perfect field which admits a henselian valuation. Assume
that the extension field $L$ of $K$ admits a $K$-rational
place. Then $K$ is existentially closed in~$L$.
\end{theorem}

For a general field $K$, we do not know whether conditions 1), 2) and 3)
of this theorem are equivalent. Note that in general, not every field
$L$ as in 3) is embeddable in a power series field $K((G))$ (as we are
not admitting non-trivial factor systems here; cf.\ [KA]). Hence it is
an interesting question whether 2) and 3) are {\it always} equivalent.

A field $K$ is existentially closed in an extension field $L$ if it
is existentially closed in every finitely generated subextension $F$
in $L$. If $L$ admits a $K$-rational place $P$, then every such function
field $F$ admits a $K$-rational place, namely, the restriction of $P$.
Hence, condition 3) of the foregoing theorem is equivalent to the
following condition on $K$:
\sn
{\bf (RP=EC)} \ {\it If an algebraic function field $F|K$ admits a
rational place, then $K$ is existentially closed in $F$.}
\sn
Note that in contrast to condition 2) of Theorem~\ref{lfecKt}, we are
not restricting our condition to function fields in one variable here.

By Theorem~\ref{lfecKt}, every field $K$ which satisfies (RP=EC)
is large. Let us see what we can say about the converse. Take a
function field $F|K$ with a rational place $P$ which admits local
uniformization. That is, $F|K$ admits a model on which $P$ is centered
at a simple $K$-rational point. By Theorem~\ref{exratpl}, $F$ also
admits a $K$-rational place $Q$ of maximal rank. By Theorem~\ref{slu},
$Q$ admits smooth local uniformization. Hence by Theorem~\ref{densrplu},
$F|K$ also admits a rational discrete place. If $K$ is large, then it
follows from Theorem~\ref{lfecKt} that $K$ is existentially closed in
$F$. This proves the following well known result:
\begin{theorem}                             \label{kecF1}
Let $K$ be a large field and $F|K$ an algebraic function field. If there
is a rational place of $F|K$ which admits local uniformization, then $K$
is existentially closed in~$F$.
\end{theorem}

As an immediate consequence, we obtain:
\begin{theorem}
Assume that all rational places of arbitrary function fields admit local
uniformization. Then every large field satisfies (RP=EC), and the three
conditions of Theorem~\ref{pcond} are equivalent, for arbitrary
fields~$K$.
\end{theorem}

Theorem~\ref{kecF1} together with Theorem~\ref{slu} implies:
\begin{corollary}                           \label{lrdraec}
Let $K$ be a large field and $F|K$ an algebraic function field. If there
is a rational discrete or a rational Abhyankar place of $F|K$, then $K$
is existentially closed in~$F$.
\end{corollary}
\sn
For the case of $F|K$ admitting a rational discrete place $P$, the
assertion is already contained in Theorem~\ref{lfecKt}.

\begin{remark}
In [ER], Yuri Ershov proves the following: {\it If $K$ admits a
henselian valuation and $V$ is an algebraic variety over $K$ with
function field $F$ which admits a rational generalized discrete place
$P$, then $V$ has a simple rational point.} It then follows by
Theorem~\ref{exratpl} and Corollary~\ref{lrdraec} that $K$ is
existentially closed in $F$. Here, ``generalized discrete'' means that
the value group is the lexicographic product of copies of $\Z$. Ershov
also observes that if $\chara K=0$ or a weak form of local
uniformization holds in positive characteristic, then one does not need
any condition on the value group of the rational place.
\end{remark}

To conclude with, let us state the converse of our above results. The
following is a generalization of the lemma on p.~190 of [K--P]:
\begin{theorem}                               \label{exrp}
Let $F|K$ be an algebraic function field such that $K$ is existentially
closed in $F$. Take any elements $z_1,\ldots,z_n\in F$. Then there are
infinitely many rational places of $F|K$ of maximal rank which are
finite on $z_1,\ldots,z_n\,$.
\end{theorem}
Note the analogy between this theorem and Theorem~\ref{exratpl}.

%
%ÄÄÄÄÄÄÄÄÄÄÄÄÄÄÄÄÄÄÄÄÄÄÄÄÄÄÄÄÄÄÄÄÄÄÄÄÄÄÄÄÄÄÄÄÄÄÄÄÄÄÄÄÄÄÄÄÄÄÄÄÄÄÄÄÄÄÄÄÄÄÄ
%
\subsection{The key ingredient for the proof of the main theorem}
The key ingredient in our proofs of Theorem~\ref{MKP} and
Theorem~\ref{MKPZ} is our generalization of the Ax--Kochen--Ershov
Theorem to the class of all tame fields. A \bfind{tame field} is a
henselian valued field $(K,v)$ for which the ramification field $K^r$ of
the normal extension $(K\sep|K,v)$ is algebraically closed. Here,
$K\sep$ denotes the separable algebraic closure. It follows from the
definition that for a tame field $(K,v)$, $K\sep$ is algebraically
closed, i.e., $K$ is perfect. For further basic properties of tame
fields, see Section~\ref{sectprt}.

By $P_v$ we denote the place associated with $v$. The following theorem
was proved in [K1] (and will be published in [K2], [K7]):

\begin{theorem}                             \label{tameAKE}
Every tame field $(K,v)$ satisfies the following Ax-Kochen-Ershov
principle:\sn
If $(K,v)\subset (F,v)$ is an extension of valued fields, $KP_v$ is
existentially closed in $FP_v$ (in the language of fields), and $vK$ is
existentially closed in $vF$ (in the language of ordered groups), then
$(K,v)$ is existentially closed in $(F,v)$ (in the language of valued
fields).
\end{theorem}
\n
For the meaning of ``existentially closed'' in the setting of valued
fields and of ordered abelian groups, see [K--P].

Note that for the proof of Corollary~\ref{cor2} we would only need
Abraham Robinson's theorem on the model completeness of the theory of
algebraically closed valued fields (cf.\ [RO]). But we need the above
theorem in order to obtain assertions (a) and (b) in Theorems~\ref{MKP},
\ref{MKPZ} and~\ref{MTdropd}.

%
%ÄÄÄÄÄÄÄÄÄÄÄÄÄÄÄÄÄÄÄÄÄÄÄÄÄÄÄÄÄÄÄÄÄÄÄÄÄÄÄÄÄÄÄÄÄÄÄÄÄÄÄÄÄÄÄÄÄÄÄÄÄÄÄÄÄÄÄÄÄÄÄ
%
\section{Some preliminaries}
%
%Let $(K,v)$ be a henselian valued field and denote by $p$ the
%characteristic exponent of its residue field $\ovl{K}$. Then $(K,v)$ is
%said to be a \bfind{tame field} if $(\tilde{K}|K,v)$ is a tame
%extension, and to be a \bfind{separably tame field} if $(K\sep|K,v)$
%is a tame extension.
%
For basic facts from valuation theory, see [EN], [RI], [W], [Z--SA],
[K2].

%
%ÄÄÄÄÄÄÄÄÄÄÄÄÄÄÄÄÄÄÄÄÄÄÄÄÄÄÄÄÄÄÄÄÄÄÄÄÄÄÄÄÄÄÄÄÄÄÄÄÄÄÄÄÄÄÄÄÄÄÄÄÄÄÄÄÄÄÄÄÄÄÄ
%
\subsection{Valuation independence}
For the easy proof of the following theorem, see [BO], Chapter VI,
\S10.3, Theorem~1, or [K2].
\begin{lemma}                                      \label{prelBour}
Let $L|K$ be an extension of fields and $v$ a valuation on $L$ with
associated place $P_v\,$. Take elements $x_i,y_j \in L$, $i\in I$, $j\in
J$, such that the values $vx_i\,$, $i\in I$, are rationally independent
over $vK$, and the residues $y_jP_v$, $j\in J$, are algebraically
independent over $KP_v$. Then the elements $x_i,y_j$, $i\in I$, $j\in
J$, are algebraically independent over $K$. Moreover,
\begin{eqnarray*}
vK(x_i,y_j\mid i\in I,j\in J) & = & vK\oplus\bigoplus_{i\in I}
\Z vx_i \;, \\
K(x_i,y_j\mid i\in I,j\in J)P_v & = & KP_v\,(y_jP_v\mid j\in J)\;.
\end{eqnarray*}
The valuation $v$ on $K(x_i,y_j\mid i\in I,j\in J)$ is
uniquely determined by its restriction to $K$, the values $vx_i$ and
the residues $y_jP_v$.
\parm
Conversely, let $(K,v)$ be any valued field, $\alpha_i$ values in
some ordered abelian group extension of $vK$, and $\xi_j$ elements in
some field extension of $KP_v\,$. Then there exists an extension of $v$
to the purely transcendental extension $K(x_i,y_j\mid i\in I,j\in J)$
such that $vx_i=\alpha_i$ and $y_jP_v=\xi_j\,$.
\end{lemma}

\begin{corollary}                              \label{fingentb}
Let $L|K$ be an extension of finite transcendence degree,
and $v$ a valuation on $L$. Then
\begin{equation}                            \label{wtdgeq}
\trdeg L|K\>\geq\>\trdeg LP_v|KP_v\,+\,\dim_\Q\,\Q\otimes (vL/vK)\;.
\end{equation}
If in addition $L|K$ is a function field and if equality holds in
(\ref{wtdgeq}), then $vL/vK$ and $LP_v|KP_v$ are finitely
generated.
\end{corollary}
\begin{proof}
Choose elements $x_1,\ldots,x_{\rho},y_1,\ldots,y_{\tau}\in L$ such that
the values $vx_1,\ldots,vx_{\rho}$ are rationally independent over $vK$
and the residues $y_1P_v,\ldots,y_{\tau}P_v$ are algebraically
independent over $KP_v$. Then by the foregoing lemma, $\rho+\tau\leq
\trdeg L|K$. This proves that $\trdeg LP_v|KP_v$ and the rational rank
of $vL/vK$ are finite. Therefore, we may choose the elements $x_i,y_j$
such that $\tau=\trdeg LP_v|KP_v$ and $\rho=\dim_{\Q}\,\Q\otimes
(vL/vK)$ to obtain inequality (\ref{wtdgeq}).

Assume that this is an equality. This means that for $L_0:=K(x_1,\ldots,
x_{\rho},y_1,\ldots,y_{\tau})$, the extension $L|L_0$ is algebraic. Since
$L|K$ is finitely generated, it follows that this extension is finite.
This yields that $vL/vL_0$ and $LP_v| L_0P_v$ are finite (cf.\ [EN],
[RI] or [BO]). Since already $v L_0/v K$ and $L_0P_v|KP_v$ are finitely
generated by the foregoing lemma, it follows that also $vL/vK$
and $LP_v|KP_v$ are finitely generated.
\end{proof}

The proof of the following fact will be published in [K8]:
\begin{proposition}                         \label{trdegmi}
Let $L|K$ be an extension of finite transcendence degree, and $v$ a
nontrivial valuation on $L$. If $\trdeg LP_v|KP_v\geq 1$ or $\dim_\Q\,
\Q\otimes (vL/vK)\geq 1$, then $(L,v)$ admits an immediate extension of
infinite transcendence degree.
\end{proposition}

%
%ÄÄÄÄÄÄÄÄÄÄÄÄÄÄÄÄÄÄÄÄÄÄÄÄÄÄÄÄÄÄÄÄÄÄÄÄÄÄÄÄÄÄÄÄÄÄÄÄÄÄÄÄÄÄÄÄÄÄÄÄÄÄÄÄÄÄÄÄÄÄÄ
%
\subsection{Tame fields}             \label{sectprt}
We will now discuss basic properties of tame fields. Let $(K,v)$ be a
valued field. Recall that we set $p=\chara KP_v$ if this is positive,
and $p=1$ if $\chara K=0$. By ramification theory, $K\sep|K^r$ is a
$p$-extension. Hence if $(K,v)$ is a henselian field of residue
characteristic $\chara KP_v=0$, then this extension is trivial. Since
then also $\chara K=0$, it follows that $K^r=K\sep=\tilde{K}$, the
algebraic closure of $K$. Therefore,
\begin{lemma}                               \label{tamerc0}
Every henselian field of residue characteristic 0 is a tame field.
\end{lemma}

Suppose that $K_1|K$ is a subextension of $K^r|K$. Then $K_1^r=K^r$.
This proves:
\begin{lemma}                               \label{alget}
Every algebraic extension of a tame field is again a tame field.
\end{lemma}

%If in addition $K_1|K$ is finite, then it satisfies the fundamental
%equality. Thus, if $K^r=\tilde{K}$, then $(K,v)$ is a defectless field.
%So we note:
%\begin{lemma}                               \label{thdp}
%Every tame field is henselian, defectless and perfect.
%\end{lemma}
%In general, infinite extensions of defectless fields need
%not be defectless fields.

A valued field $(K,v)$ is called \bfind{algebraically maximal} if it
admits no proper immediate algebraic extension. Since the henselization
is an immediate algebraic extension, every algebraically maximal field
is henselian. We give a characterization for tame fields
(the proof will be published in [K2], [K7]):
\begin{lemma}                    \label{tame}
The following assertions are equivalent:\n
1)\ \  $(K,v)$ is a tame field,\n
2)\ \  $(K,v)$ is algebraically maximal, $vK$ is $p$-divisible
and $KP_v$ is perfect.
\sn
If in addition $\chara K=\chara KP_v$, then the above assertions
are also equivalent to
\n
3)\ \  $(K,v)$ is algebraically maximal and perfect.
\end{lemma}

\begin{corollary}                           \label{cortame}
Assume that $\chara K=\chara KP_v$. Then every maximal immediate
algebraic extension of the perfect hull of $(K,v)$ is a tame field.
\end{corollary}

%Every henselian defectless field is algebraically maximal, but the
%converse does not hold in general. However, the corollary shows that for
%perfect valued fields $(K,v)$ with $\chara K=\chara Kv$, the properties
%``algebraically maximal'' and ``henselian defectless'' are equivalent.

\pars
Assume that $\chara K=p>0$, and let $K^{1/p^{\infty}}$ denote the
perfect hull of $K$. There is a unique extension of every valuation $v$
from $K$ to $K^{1/p^{\infty}}$, which we will again denote by $v$. The
value group $vK^{1/p^{\infty}}$ is the $p$-divisible hull of $vK$, and
the residue field $K^{1/p^{\infty}}P_v$ is the perfect hull of $KP_v\,$.
But even if $\chara K\ne p$, Section~2.3 of [K5] shows that it is easy
to construct an algebraic extension $(K',v)$ such that
\sn
$\bullet$ \ $vK'=\frac{1}{p^{\infty}}vK$ (the $p$-divisible
hull of $vK$), and
\sn
$\bullet$ \ $K'P_v=(KP_v)^{1/p^{\infty}}$ (the perfect hull of $KP_v$).
\sn
If these two assertions hold, then by Lemma~\ref{tame}, every maximal
immediate algebraic extension $(L,v)$ of $(K',v)$ is a tame field. We
have proved:

\begin{proposition}                         \label{ete}
For every valued field $(K,v)$, there is an algebraic extension $(L,v)$
which is a tame field and satisfies
\begin{equation}                            \label{etevgrf}
vL\>=\>\frac{1}{p^{\infty}}vK \ \mbox{\ \ and\ \ } \
LP_v=(KP_v)^{1/p^{\infty}}\;.
\end{equation}
\end{proposition}

\pars
The following is a crucial lemma in the theory of tame fields. It was
proved in [K1] (the proof will be published in [K2], [K7]).
\begin{lemma}                                     \label{trac}
Let $(L,v)$ be a tame field and $K\subset L$ a relatively algebraically
closed subfield. If in addition $LP_v|KP_v$ is an algebraic extension,
then $(K,v)$ is also a tame field and moreover, $vL/vK$ is torsion free
and $KP_v=LP_v$.
\end{lemma}

%
%Ä - Ä Ä Ä Ä Ä Ä Ä Ä Ä Ä Ä Ä Ä Ä Ä Ä Ä Ä Ä Ä Ä Ä Ä Ä Ä Ä Ä Ä Ä Ä Ä Ä Ä
%
\section{Proof of the main theorems}
%
%
%ÄÄÄÄÄÄÄÄÄÄÄÄÄÄÄÄÄÄÄÄÄÄÄÄÄÄÄÄÄÄÄÄÄÄÄÄÄÄÄÄÄÄÄÄÄÄÄÄÄÄÄÄÄÄÄÄÄÄÄÄÄÄÄÄÄÄÄÄÄÄÄ
%
\subsection{Proof of Theorem~\ref{MKP}}           \label{sectpr1}  %
%
%Theorem~\ref{MKP}\n
Assume that $F|K$ is an algebraic function field of transcendence degree
$n$, $\Pzero$ a place on $K$, $Q\in S(F|K\,;\,\Pzero)$, and $a_1,\ldots,
a_m\in F$. We set
\[d\>:=\> \dim Q\mbox{ \ \ and \ \ } r\>:=\>\mbox{\rm rr}\,Q\;.\]
Then we choose $y_1,\ldots,y_d\in F$ such that $y_1Q,\ldots, y_dQ$ form
a transcendence basis of $FQ|K\Pzero\,$. Further, we choose $x_1,\ldots,
x_r\in F$ such that the values $v_Q x_1,\ldots,v_Qx_r$ form a maximal set
of rationally independent elements in $v_QF$ modulo $v_{\Pzero}K$.
According to Lemma~\ref{prelBour}, the elements $x_1,\ldots,x_r,
y_1,\ldots,y_d$ are algebraically independent over $K$. We take $K_0$ to
be the rational function field $K(x_1,\ldots,x_r,y_1,\ldots,y_d)$.

By Proposition~\ref{ete} we choose an algebraic extension
$(L,Q)$ of $(F,Q)$ which is a tame field and satisfies (\ref{etevgrf}),
for $F$ in the place of $K$.
%
%be a maximal immediate algebraic extension of
%$(F^{1/p^{\infty}},Q)$. By virtue of Lemma~\ref{tame}, $(L,Q)$
%is a tame field. We have that $v_Q L=v_Q F^{1/p^{\infty}}=
%\frac{1}{1/p^{\infty}}v_Q F$ (the $p$-divisible hull of $v_Q F$), and
%that $LQ=F^{1/p^{\infty}}Q=(FQ)^{1/p^{\infty}}$ (the perfect hull of
%$FQ$).
%$v_Q L=\frac{1}{1/p^{\infty}}v_Q F$ and $LQ=(FQ)^{1/p^{\infty}}$.
%
By construction of $K_0\,$, we have that $v_Q L/v_QK_0$ is a
torsion group and $LQ|K_0Q$ is algebraic.

Let $K'$ be the relative algebraic closure of $K_0$ in $L$, and let
$Q'$ be the restriction of $Q$ to $K'$. According to Lemma~\ref{trac},
$(K',Q')$ is a tame field with
\begin{equation}                            \label{imme}
K'Q' = LQ \mbox{\ \ and\ \ } v_{Q'} K' = v_Q L\;.
\end{equation}
Hence $(K',Q')$ is existentially closed in $(L,Q)$ by
Theorem~\ref{tameAKE}.

Since the tame field $K'$ is perfect, the algebraic function field
$K'.F|K'$ is separably generated. Therefore, we can write $K'.F=
K'(t_1,\ldots,t_k, y)$, where $k = n - (d+r)$, the elements
$t_1,\ldots,t_k$ are algebraically independent over $K'$, and $y$ is
separable algebraic over $K'(t_1,\ldots,t_k)$. Let $f\in
K'[t_1,\ldots,t_k,Y]$ be the irreducible polynomial of $y$ over
$K'[t_1,\ldots,t_k]$. For $\mbb{t}=(t_1,\ldots,t_k)$, we then have
\[f(\mbb{t},y)=0\;\;\;\mbox{\ \ and\ \ }\;\;\;
\frac{\displaystyle\partial f}{\displaystyle\partial Y}
(\mbb{t},y)\not= 0\;.\]
In view of (\ref{imme}), we can choose $a'_1,\ldots,a'_m\in K'$ such
that
\[a'_iQ \>=\> a_iQ \mbox{ \ \ and \ \ } v_Q a'_i \>=\> v_Q a_i
\mbox{ \ \ for \ \ } 1\leq i\leq m\;.\]
We write the elements $a_i$ as follows:
\[a_i = \frac{g_i(\mbb{t},y)}{h_i(\mbb{t})}
\mbox{\ \ \ for\ \ \ } 1\leq i \leq m\;,\]
where $g_i$ and $h_i$ are polynomials over $K'$, with $h_i(\mbb{t})\ne
0$. Since $(K',Q')$ is existentially closed in $(L,Q)$, there exist
elements
\[t'_1\,,\ldots,\,t'_k\,,\, y' \in K'\]
such that
\[\begin{array}{ll}
   (i) & f(\mbb{t}',y') = 0 \mbox{\ \ and\ \ }
         \frac{\displaystyle\partial f}{\displaystyle\partial Y}
         (\mbb{t}',y')\not= 0\;,\\[0.3cm]
   (ii) & h_i(\mbb{t}') \not= 0
         \mbox{\ \ for\ \ } 1\leq i \leq m\;,\\[0.3cm]
   (iii) & \frac{\displaystyle g_i(\mbb{t}',y')}
         {\displaystyle h_i(\mbb{t}')}Q = a'_iQ \mbox{\ \ and\ \ }
         v_Q\,\frac{\displaystyle g_i(\mbb{t}',y')}
         {\displaystyle h_i(\mbb{t}')} = v_Q a'_i
         \mbox{\ \ for\ \ }1\leq i \leq m\;,
\end{array}\]
since these assertions are true in $L$ for $\mbb{t},y$ in the place
of $\mbb{t}',y'$.

\par\smallskip
Now let $K_1$ be the subfield of $K'$ which is generated over $K$ by the
following elements:\sn
$\bullet \ \ x_1,\ldots,x_r,y_1,\ldots,y_d\,$,\n
$\bullet \ \ a'_1,\ldots,a'_m, t'_1,\ldots,t'_k, y'\,$,\n
$\bullet \ \ $the coefficients of $f$, $g_i$ and $h_i$ for $1\leq i
\leq m$.

\sn
Let $P_1$ denote the restriction of $Q$ to $K_1$. We note that
$K_1$ is a finite extension of $K_0$. Hence according to
Corollary~\ref{fingentb}, $v_{P_1} K_1$ is a subgroup of $v_Q L$,
finitely generated over $v_{\Pzero}K$ with $v_{P_1} K_1/v_{\Pzero}K$ of
rational rank $r$. Similarly, $K_1P_1$ is a subfield of $LQ$ and
finitely generated of transcendence degree $d$ over $K\Pzero\,$.

\pars
At this point we may forget about the field $L$ and its place $Q$.
Starting from $(K_1,P_1)$ we will construct some henselian extension
of $(K_1,P_1)$ which will contain an isomorphic copy of $F$. The
construction will be done in such a way that the restriction of the
place to the embedded copy of $F$ will satisfy the assertion of the
theorem.

\pars
Choose $d_1$ and $r_1$ as in the assumption of Theorem~\ref{MKP}.
We take $d_1-d$ elements $y_{d+1},\ldots,y_{d_1}$, algebraically
independent over $K_1\,$, and set $K_2:=K_1(y_{d+1},\ldots,y_{d_1})$.
We extend $P_1$ to a place $P_2$ on $K_2$ such that the value group
does not change and the residue field $K_2P_2$ becomes a purely
transcendental extension of $K_1P_1$ of transcendence degree $d_1-d$.
This can be done by an application of Lemma~\ref{prelBour}.

\pars
Next we adjoin $r_1-r$ elements $x_{r+1},\ldots,x_{r_1}\,$,
algebraically independent over $K_2\,$. We assume that an arbitrary
ordering on $\frac{1}{p^{\infty}}v_Q F\oplus \bigoplus_{r_1-r}\Z$ has
been fixed. We take $\alpha_1,\ldots,\alpha_{r_1-r}$ to be generators of
that group over $v_{P_2}K_2=v_Q K_1\,$; then $\alpha_1,\ldots,
\alpha_{r_1-r}$ are rationally independent over $\frac{1}{p^{\infty}}v_Q
F$. By Lemma~\ref{prelBour}, there is an extension $P_3$ of $P_2$ to
$K_3:= K_2(x_{r+1},\ldots,x_{r_1})$ such that $v_{P_3}x_{r+i}=
\alpha_i\,$ for $1\leq i\leq r_1-r$, with $v_{P_3}K_3= v_{P_2}K\oplus
\bigoplus_{r_1-r} \Z\subseteq \frac{1}{p^{\infty}}v_Q F\oplus
\bigoplus_{r_1-r} \Z$ and $K_3P_3=K_2P_2$.

\pars
By construction, $(K_3,P_3)$ satisfies:\sn
$\bullet$ \ $\trdeg K_3|K\>=\> d+r+(r_1-r)+(d_1-d)\>=\>d_1+r_1\,$,\n
$\bullet$ \ $K_3P_3|K\Pzero$ is finitely generated, with
$\trdeg K_3P_3|K\Pzero= d+(d_1-d)\>=\>d_1\,$,\n
$\bullet$ \ $v_{P_3}K_3/v_{\Pzero}K$ is finitely generated, with
$\mbox{rr} \ v_{P_3}K_3/v_{\Pzero}K= r+(r_1-r)\>=\>r_1\,$.\sn

\pars
If $d_1+r_1\geq 1$, then $\trdeg K_3|K \geq 1$, and
Proposition~\ref{trdegmi} shows that $(K_3,P_3)$ admits an immediate
extension of transcendence degree $n-(d_1+r_1)$. If $d_1+r_1=0$, then
our construction yields $K_3=K_1$ which is a finite extension of
$K_0=K$. In this case, the existence of such an immediate extension is
guaranteed by the additional assumption at the end of our theorem. Now
we pick any transcendence basis of this extension and take $K_4$ to be
the subextension which it generates over $K_3$. Restricting the place to
the so obtained field, we get an immediate extension $(K_4,P_4)$ of
$(K_3,P_3)$.

\pars
Now we take $(K_5,P_5)$ to be the henselization of $(K_4,P_4)$. It
remains to show that $F$ can be embedded in $K_5$ over $K$. Then $P_5$
will induce a place $P$ on $F$ which satisfies the assertions of our
theorem. In fact, we find an embedding of $K_1 .F$ over $K_1$ in $K_5$
as follows.

\pars
We choose elements $t_1^*,\ldots,t_k^*\in K_5$, algebraically
independent over $K_1$, so close to $t'_1,\ldots,t'_k$ that by the
Implicit Function Theorem (which holds in every henselian field,
cf.\ [P--Z], Theorem 7.4) we can find $y^*\in K_5$ satisfying
$f(\mbb{t}^*,y^*)=0$ and being so close to $y'$ that in addition,
(ii) and (iii) hold for $\mbb{t}^*,y^*$ in the place of $\mbb{t}',y'$
and $P_5$ in the place of $Q$. Since $\mbb{t}',y'$ satisfy (ii) and
(iii), and these conditions define an open set in the valuation
topology, such elements $t_1^*,\ldots,t_k^*,y^*$ can be found in
$K_5\,$. The fact that $t_1^*,\ldots,t_k^*$ can even be chosen to be
algebraically independent over $K_1$ follows from the choice of the
transcendence degree of $K_5$ over $K_1$ (which is
$(d_1-d)+(r_1-r)+n-(d_1+r_1)= n-(d+r)=k\,$), and the fact that for any
intermediate field $K_1\subset K'_1\subsetuneq K_5$ which is relatively
algebraically closed in $K_5$, the elements of $K_5\setminus K'_1$ lie
dense in $K_5\,$. Applying this fact inductively yields the result.

\pars
Now $t_i\mapsto t_i^*$ ($1\leq i\leq k$) and $y\mapsto y^*$ defines an
embedding of $K_1.F$ over $K_1$ in $K_5\,$. We identify $F$ with its
image in $K_5$ and take $P$ to be the restriction of $P$ to $F$. By
construction, $K_4(\mbb{t}^*,y^*)$ is a finite algebraic extension of
$F$, having the purely transcendental extension $K_4(\mbb{t}^*,y^*)P_5=
K_3P_3$ of $K_1P_1$ of transcendence degree $d_1-d$ as its residue
field, and the $(r_1-r)$-extension $v_{P_5}K_4(\mbb{t}^*,y^*)=v_{P_3}
K_3$ of $v_{P_1}K_1$ as its value group. Further, it follows that
$[K_3P_3:FP]$ is finite and therefore, $FP|K\Pzero$ is finitely
generated of transcendence degree $d_1\,$. It also follows that
$(v_{P_3}K_3:v_PF)$ is finite and therefore, $v_PF/v_{\Pzero}K$ is
finitely generated of rational rank $r_1\,$. Thus, $FP$ and $v_PF$
satisfy conditions (a) and (b) of the theorem.

%\pars
%By construction, $v_PF$ is a subgroup of $v_{P_3}K_3$ of the same
%rational rank, and ; thus,
%$(v_{P_3}K_3:v_PF)$ is finite and $v_PF$ is finitely generated.
%Similarly, $FP$ is a subfield of $K_3P_3$ of the same transcendence
%degree over $K$, and $K_3P_3$ is finitely generated over $K$; thus,

\pars
Finally, we have to check the conditions on the elements
$a_i\,$. After identifying $K_1.F$ with its image in $K_5$, we have that
\[a_i\>=\>\frac{g_i(\mbb{t}^*,y^*)}{h_i(\mbb{t}^*)}\;.\]
Now the result follows from assertion (iii) (with $P_5$ replaced by $Q$)
for $\mbb{t}^*,y^*$, together with $a'_iP_5=a'_iQ=a_iQ$ and $v_{P_5}
a'_i= v_Qa'_i=v_Qa_i$ ($1\leq i\leq m$).                       \QED

%
%ÄÄÄÄÄÄÄÄÄÄÄÄÄÄÄÄÄÄÄÄÄÄÄÄÄÄÄÄÄÄÄÄÄÄÄÄÄÄÄÄÄÄÄÄÄÄÄÄÄÄÄÄÄÄÄÄÄÄÄÄÄÄÄÄÄÄÄÄÄÄÄ
%
\subsection{Proof of Theorem~\ref{MKPZ}}           \label{sectpr2}
%
%Theorem~\ref{MKPZ}\n
For the {\bf proof of Theorem~\ref{MKPZ}}, we modify the above proof in
the following way. We replace $(L,Q)$ by a maximal algebraic extension
still having $F^{1/p^{\infty}}Q=(FQ)^{1/p^{\infty}}$ as its residue
field. Such an extension will have a divisible value group (cf.\
Section~2.3 of [K5]). The new $(L,Q)$ is again a tame field, by
Lemma~\ref{tame}.

Let us first assume that $\Pzero$ is trivial. Then we take $K'$ to be the
relative algebraic closure of $K(x_1,y_1,\ldots,y_d)$ in $L$, and $Q'$
the restriction of $Q$. By Lemma~\ref{trac}, $(K',Q')$ is a tame field
with $K'Q'=LQ= (FQ)^{1/p^{\infty}}$ and $v_{Q'}K'=\Q vx_1\,$. We choose
$a'_1,\ldots,a'_m\in K'$ such that
\[a'_iQ' \>=\> a_iQ \mbox{ \ \ for \ \ } 1\leq i\leq m\;.\]
As $v_{Q'}K'$ may be smaller than $v_Q L$, it may not be possible to
choose the $a'_i$ such that also $v_{Q'}a'_i=v_Qa_i\,$. Since a
divisible ordered abelian group is existentially closed in every ordered
abelian group extension, we can again apply Theorem~\ref{tameAKE}. But
we have to replace (iii) by
\[(iii) \ \ \frac{\displaystyle g_i(\mbb{t}',y')} {\displaystyle
h_i(\mbb{t}')}Q = a'_iQ \mbox{\ \ for\ \ }1\leq i \leq m\;.\]
Therefore, we cannot preserve information about the values $v_Qa_i\,$.
On the other hand, we gain the freedom to extend the value group of
$(K_1,P_1)$ (which is a finite extension of $K(x_1,y_1,\ldots,y_d)$ and
thus has value group isomorphic to $\Z$) by $r_1-1$ new copies of
$\Z$, where $r_1$ can be chosen freely between $1$ and $n-d_1\,$.

\parm
In the case of $\Pzero$ being non-trivial, we take $K'$ to be the
relative algebraic closure of $K(y_1,\ldots,y_d)$ in $L$ and proceed as
above. In this case, the value group of $(K_1,P_1)$ is a finite
extension of $v_{\Pzero}K$. In the subsequent construction, we extend it
by $r_1$ new copies of $\Z$, where $r_1$ can be chosen freely between
$0$ and $n-d_1\,$.                                       \QED

%
%ÄÄÄÄÄÄÄÄÄÄÄÄÄÄÄÄÄÄÄÄÄÄÄÄÄÄÄÄÄÄÄÄÄÄÄÄÄÄÄÄÄÄÄÄÄÄÄÄÄÄÄÄÄÄÄÄÄÄÄÄÄÄÄÄÄÄÄÄÄÄÄ
%
\subsection{Proof of Theorem~\ref{MTdropd}}            \label{sectpr3}
In view of Theorem~\ref{MKP}, we only have to show that there is a
zero-dimensional place $P$ which satisfies (\ref{valeq}) and for which
$v_P F$ is equal to the group $\Gamma$ which is described in assertion
(b) of Theorem~\ref{MTdropd}, and is finitely generated over
$v_{\Pzero}K$. First, we use Theorem~\ref{MKP} to find a place $Q_1$
such that
\sn
-- \ $v_{Q_1}a_i=v_Qa_i$ for $1\leq i \leq m$,
\n
-- \ $v_{Q_1}F$ is a subgroup of the $p$-divisible
hull of $v_Q F$, finitely generated over $v_{\Pzero}K$,
\n
-- \ $FQ_1|K\Pzero$ is finitely generated.
\sn
Now we have to construct a zero-dimensional place from $Q_1\,$. The idea
is to find a zero-dimensional place on the function field $FQ_1|K\Pzero$
and then to compose it with $Q_1$.

We take $K'$ to be the algebraic closure of $K\Pzero$ in the algebraic
closure of $FQ_1$. Since $K'$ is algebraically closed, it is
existentially closed in $F':=K'.(FQ_1)\,$. Hence, we can apply
Theorem~\ref{exrp} to the function field $F'|K'$. This gives us a
$K'$-rational place $Q'_2$ of $K'.(FQ_1)$. Its restriction to $FQ_1$ is
a zero-dimensional place of $FQ_1|K\Pzero\,$. We use Theorem~\ref{MKPZ}
to change it to a discrete zero-dimensional place $Q_2$ of
$FQ_1|K\Pzero\,$. Now $Q_1Q_2$ is indeed a zero-dimensional place
of~$F$. Its value group $v_{Q_1Q_2}F$ contains $v_{Q_2}FQ_1=\Z$ as a
convex subgroup, and $v_{Q_1Q_2}F/v_{Q_2}FQ_1$ is isomorphic to
$v_{Q_1}F$. We set $P=Q_1Q_2\,$, so $\Gamma:=v_PF=v_{Q_1Q_2}F$ is
as described in assertion (b) of Theorem~\ref{MTdropd}. Since $Q_2$ is
trivial on $K\Pzero$, we have that $v_{Q_1Q_2}K =v_{\Pzero}K$ and that
$v_PF/v_{\Pzero}K=v_{Q_1Q_2}F/v_{Q_1Q_2}K$ still has convex subgroup
$v_{Q_2}FQ_1=\Z$ such that the quotient is $v_{Q_1}F$. Hence, also
$v_PF/v_{\Pzero}K$ is finitely generated.

However, we have to be more careful in order to satisfy the condition on
the values. We choose a $\Z$-basis $\gamma_1,\ldots,\gamma_\ell$ of the
group generated by the values $v_Qa_1\,,\ldots,v_Qa_m\,$, and elements
$b_1,\ldots,b_\ell\in F$ such that $v_{Q_1}b_i=\gamma_i\,$. Since these
values are rationally independent and since $v_{Q_1} F$ is a quotient of
$v_{Q_1Q_2} F$ by a convex subgroup, it follows that sending
$v_{Q_1}b_i$ to $v_{Q_1Q_2} b_i$ induces an order preserving embedding
$\iota$ of $\Gamma$ in $v_{Q_1Q_2} F$. We have to choose $Q_2$ in such a
way that $\iota v_{Q_1}a_i=v_{Q_1Q_2} a_i$ for $1\leq i \leq m$. By our
choice of the $\gamma_i$ there are integers $e_{i,j}$ such that
\[v_{Q_1}a_i\>=\>\sum_{j=1}^{\ell} e_{i,j} \gamma_j\>=\>
v\prod_{j=1}^{\ell} b_j^{e_{i,j}}\;,\]
whence
\[v_{Q_1} a'_i\>=\>0\;\;\;\mbox{ for }\;\;\;a'_i\>:=\> a_i^{-1}
\prod_{j=1}^{\ell} b_j^{e_{i,j}}\;.\]
That is, $a'_iQ_1\ne 0$, and by Theorem~\ref{exrp}, we can choose
$Q_2$ such that $a'_iQ_1Q_2\ne 0$ for $1\leq i \leq m$. That gives us
$v_{Q_1Q_2} a'_i=0$ and consequently,
\[v_{Q_1Q_2} a_i\>=\>v_{Q_1Q_2} \prod_{j=1}^{\ell} b_j^{e_{i,j}}
\>=\>\sum_{j=1}^{\ell} e_{i,j} v_{Q_1Q_2} b_j\>=\>
\sum_{j=1}^{\ell} e_{i,j} \iota v_{Q_1} b_j\>=\>
\iota \sum_{j=1}^{\ell} e_{i,j} v_{Q_1} b_j\>=\>\iota v_{Q_1} a_i\;.\]
This completes the proof of Theorem~\ref{MTdropd}.

%
%Ä - Ä Ä Ä Ä Ä Ä Ä Ä Ä Ä Ä Ä Ä Ä Ä Ä Ä Ä Ä Ä Ä Ä Ä Ä Ä Ä Ä Ä Ä Ä Ä Ä Ä
%
\subsection{Proof of Theorem~\ref{densrplu}}\label{sectdrplu}
The proof is an adaptation of the last part of the proof of
Theorem~\ref{MKP}. We take a function field $F|K$ with a rational place
$Q$ which admits smooth local uniformization. Further, we take elements
$a_1,\ldots,a_m\in {\cal O}_Q\,$. Then there is an affine model of $F$
with coordinate ring $K[x_1,\ldots,x_k]$ such that $x_1,\ldots,x_k\in
{\cal O}_Q\,$, the point $(x_1Q,\ldots,x_kQ)$ is smooth and
$K$-rational, and $a_1,\ldots,a_m\in K[x_1,\ldots,x_k]$. Among the
elements $x_i$ we can choose a transcendence basis $t_1,\ldots,t_n$ and
can rewrite the original polynomial relations as polynomial relations
with polynomials $f_1,\ldots,f_\ell \in K[t_1,\ldots,t_n]
[Y_1,\ldots,Y_\ell]$, satisfied by the remaining $x_i$'s, which we now
call $y_1,\ldots,y_\ell\,$. These elements satisfy the hypothesis of the
multidimensional Hensel's Lemma, namely,
\[\left(\det\left(\frac{\partial f_i}{\partial Y_j}(y_1,\ldots,y_\ell)
\right)_{1\leq i\leq \ell\atop 1\leq j\leq\ell}\right)Q \>=\>
\det\left(\frac{\partial (f_iQ)}{\partial Y_j}(y_1 Q,\ldots,y_\ell Q)
\right)_{1\leq i\leq \ell\atop 1\leq j\leq\ell} \>\ne\>0\]
(cf.\ Sections 3 and 4 of [K4]). Here, $f_iQ$ denotes the polynomial
whose coefficients are obtained from the corresponding coefficients of
$f_i$ by an application of $Q$.

Now we take an arbitrary place $P'$ on a rational function field
$K(z_1,\ldots,z_n)$ with residue field $K$ and such that $z_1,\ldots,z_n
\in {\cal M}_{P'}\,$. In fact, we can choose $P'$ to be discrete (since
$K((z_1))$ is of infinite transcendence degree over $K(z_1)\,$), or of
maximal rank, or with any finitely generated value group of rational
rank $n$. (For all possible choices, see [K5].) We take $(F',P')$ to be
the henselization of $(K(z_1,\ldots,z_n),P')$.

The elements $t_1Q+z_1,\ldots,t_n Q+z_n$ are algebraically independent
over $K$, so $t_i\mapsto t_iQ+z_i$ induces an isomorphism from
$K(t_1,\ldots,t_n)$ onto $K(z_1,\ldots,z_n)$. This sends the polynomials
$f_i$ to polynomials $f_i^*$ with coefficients in $K[z_1,\ldots,z_n]$,
and these coefficients have the same images under $P'$ as the original
coefficients have under $Q$. Hence the new polynomials $f_i^*$ also
satisfy the hypothesis of the multidimensional Hensel's Lemma:
\[\det\left(\frac{\partial (f_i^*P')}{\partial Y_j}(y_1 Q,\ldots,
y_\ell Q) \right)_{1\leq i\leq \ell\atop 1\leq j\leq\ell} \>\ne\>0\]
with $(y_1 Q,\ldots,y_\ell Q)\in K^\ell\subseteq (F'P')^\ell$.
Therefore, by the
multidimensional Hensel's Lemma (which holds in every henselian field)
there is a common zero $(y'_1,\ldots,y'_\ell) \in (F')^\ell$ of the
$f_i^*$ such that $y'_iP'=y_iQ$, $1\leq i\leq\ell$. So the above
constructed isomorphism can be extended to an embedding of $F$ in $F'$.
We identify $F$ with its image and take $P$ to be the restriction of
$P'$ to $F$. It is discrete if $P'$ is, and of maximal rank if $P'$ is.
Since $t_iP= (t_iQ+z_i)P'=t_iQ$ and $y_iP= y'_iP'=y_iQ$, and since
$a_j\in K[t_1, \ldots,t_n,y_1,\ldots,y_\ell]$, we also find that
$a_jP=a_jQ$ for $1\leq j\leq m$. This proves our theorem.           \QED

%
%Ä - Ä Ä Ä Ä Ä Ä Ä Ä Ä Ä Ä Ä Ä Ä Ä Ä Ä Ä Ä Ä Ä Ä Ä Ä Ä Ä Ä Ä Ä Ä Ä Ä Ä
%
\subsection{Proofs for Sections~\ref{sectlf}
and~\ref{sectrp=ec}}                             \label{sectlfp}
We start with the\n
{\bf Proof of Theorem~\ref{exrp}:} \ We adapt the proof of the lemma
on p.~190 of [K--P]. Assume that $K$ is existentially closed in $F$.
It is well known and easy to prove that this implies that $F|K$ is
separable. Pick a separating transcendence basis $x_1,\ldots,x_d$ and
$y\in F$ separable algebraic over $K(x_1,\ldots,x_d)$ such that
$F=K(x_1,\ldots,x_d,y)$. Let $f\in K[x_1,\ldots,x_d,Y]$ be the
irreducible polynomial of $y$ over $K[x_1,\ldots,x_d]$. We write
\[z_i = \frac{g_i(\mbb{x},y)}{h_i(\mbb{x})}
\mbox{\ \ \ for\ \ \ } 1\leq i \leq n\;,\]
where $g_i$ and $h_i$ are polynomials over $K$, with $h_i(\mbb{t})\ne
0$. Since $x_1,\ldots,x_d,y$ satisfy
\[f(\mbb{x},y)=0\;\;\;\mbox{\ \ and\ \ }\;\;\;
\frac{\displaystyle\partial f}{\displaystyle\partial Y}
(\mbb{x},y)\not= 0\;\;\;\mbox{\ \ and\ \ }\;\;\; h_i(\mbb{t})\ne 0
\;\;(1\leq i\leq n)\]
in $F$, we infer from $K$ being existentially closed in $F$ that there
are $a_1,\ldots,a_d,b$ in $K$ such that
\[f(\mbb{a},b)=0\;\;\;\mbox{\ \ and\ \ }\;\;\;
\frac{\displaystyle\partial f}{\displaystyle\partial Y}
(\mbb{a},b)\not= 0\;\;\;\mbox{\ \ and\ \ }\;\;\; h_i(\mbb{a})\ne 0
\;\;(1\leq i\leq n).\]
As in [K--P], we consider the field $K((X_1))\ldots((X_d))$ and embed
$F$ in $L$ in such a way that $x_i$ is sent to $X_i+a_i\,$. This induces
a $K$-rational place $P$ on $F$ such that $x_iP=a_i\,$. It follows that
$z_iP\ne\infty$ $(1\leq i\leq n)$.

Now suppose that we have already $K$-rational places $P_1,\ldots,P_k$
which are finite on $z_1,\ldots,z_n\,$. Define
$z_{n+j}:=(x_1-x_1P_j)^{-1}\in F$ for $1\leq j\leq k$. By the above
there exists a place $P$ which is finite on
$z_1,\ldots,z_{n+k}\,$. It follows that $x_1P\ne x_1P_j$ and hence
$P\ne P_j$ for $1\leq j\leq k$. This shows that there are infinitely
many $K$-rational places which are finite on $z_1,\ldots,z_n\,$.    \QED

\pars
For the case of one variable, we extend Theorem~\ref{exrp} as
follows:
\begin{proposition}                         \label{infmany}
Let $F|K$ be a function field in one variable. Then $K$ is
existentially closed in $F$ if and only if $F$ admits infinitely many
$K$-rational places.
\end{proposition}
\begin{proof}
If $K$ is existentially closed in $F$, then Theorem~\ref{exrp}
shows that $F$ admits infinitely many $K$-rational places.

For the converse, assume that $F$ admits infinitely many $K$-rational
places. Suppose that $f_1,\ldots,f_n\in K[X_1,\ldots,X_m]$ have a common
zero $(a_1,\ldots, a_m)\in F^m$. If $P$ is a $K$-rational place of $F$
such that $a_iP\ne \infty$ for $1\leq i\leq m$, then $(a_1P,\ldots,a_mP)
\in K^m$ is a common zero of $f_1,\ldots,f_n$ since
$f_j(a_1P,\ldots,a_mP)=f_j (a_1,\ldots,a_m)P=0P=0$ for $1\leq j\leq n$.
Now it suffices to show that there are only finitely many $K$-rational
places $P$ of $F$ for which $a_iP=\infty$ for some $i$. But this is
clear because $a_i\mapsto \infty$ defines a unique place on $K(a_i)$
(namely, the $a_i^{-1}$-adic place), and since $F|K(a_i)$ is a finite
extension (as $a_i$ must be transcendental over $K$), this place has
only finitely many extensions to $F$.
\end{proof}

\mn
{\bf Proof of Theorem~\ref{lfecKt}:}
\sn
1) $\Leftrightarrow$ 2): \ By the foregoing proposition, (LF$''$) is
equivalent to 2).

\mn
2) $\Rightarrow$ 3): \ Since $K(t)^h$ admits a $K$-rational place, so
does every function field $F|K$ which is contained in $K(t)^h$. Every
such function field $F|K$ is a function field in one variable. Hence
by 2), $K$ is existentially closed in $F$. It follows that $K$ is
existentially closed in $K(t)^h$.

\mn
3) $\Rightarrow$ 4): \ If $K$ is existentially closed in $K(t)^h$, then
$K$ is existentially closed in $K((t))$, since this property is
transitive and the following holds:
\begin{theorem}
The field $K(t)^h$ is existentially closed in $K((t))$.
\end{theorem}
\n
This result follows from Theorem 2 in [ER]. Another proof will be given
in [K2].

\mn
4) $\Rightarrow$ 5): \ Suppose that $L$ is an extension field of $K$
which admits a discrete $K$-rational place. Then it has a local
parameter $t$ and can be embedded over $K(t)$ in $K((t))$. Hence if $K$
is existentially closed in $K((t))$, then it is also existentially
closed in $L$.

\mn
5) $\Rightarrow$ 2) is trivial.                           \QED

\bn
{\bf Proof of Proposition~\ref{hvilarge}:} \ Let $v$ be the henselian
valuation on $K$. Take any $|K|^+$-saturated elementary extension
$(K^*,v^*)$ of $(K,v)$, where $|K|^+$ denotes the successor cardinal of
the cardinality of $K$. Then $(K^*,v^*)$ is henselian. Moreover,
$v^*K^*$ will contain an element $\alpha$ which is bigger than every
element in $vK$.

Now we consider $K(t)^h$ with the valuation $w:=v_t\circ v$ which is the
composition of the $t$-adic valuation $v_t$ and the valuation $v$ on its
residue field $K=K(t)^h P_{v_t}\,$. Since $v$ is henselian, $K(t)^h$ is
also the henselization of $K(t)$ with respect to $w$.

The value $w t$ is bigger than every element in $wK= vK$.
Hence, sending $w t$ to $\alpha$ induces an order preserving isomorphism
from $w K(t)=vK\oplus\Z w t$ onto $vK\oplus\Z\alpha$. Consequently,
sending $t$ to some element $t^* \in K^*$ with $v^*t^*=\alpha$ induces a
valuation preserving embedding of $K(t)$ in $K^*$ (apply
Lemma~\ref{prelBour}). By the universal property of henselizations, this
embedding extends to an embedding of $K(t)^h$ in the henselian field
$K^*$.

Since all existential sentences are preserved by this embedding, and
since $K$ is existentially closed in $K^*$, it now follows that $K$
is existentially closed in $K(t)^h$.                      \QED

\bn
{\bf Proof of Theorem~\ref{pcond}:} \ The implication 3) $\Rightarrow$
2) is trivial, and 2) $\Rightarrow$ 1) follows from Theorem~\ref{lfecKt}
since $K((t))=k((\Z))$. Hence it remains to show the implication 1)
$\Rightarrow$ 3).

We take any field extension $L$ of $K$ which admits a $K$-rational place
$P$. We extend $P$ to the perfect hull $L^{1/p^{\infty}}$ of $L$.
Since $L^{1/p^{\infty}}P=(LP)^{1/p^{\infty}}=K^{1/p^{\infty}}$ and $K$
is perfect by assumption, we find that $L^{1/p^{\infty}}P=K$. Now we
take $(L_1,P)$ to be a maximal immediate algebraic extension of
$(L^{1/p^{\infty}},P)$. Hence, we still have $L_1P=K$. By
Corollary~\ref{cortame}, $(L_1,P)$ is a tame field. By adjoining
suitable $n$-th roots of elements in $L_1\,$, we can further extend to a
valued field $(L_2,P)$ such that $v_P L_2$ is divisible and $L_2P=K$. By
Lemma~\ref{alget}, also $(L_2,P)$ is a tame field.

Now take any $x\in L_2$ which is transcendental over $K$ and let $L_0$
be the relative algebraic closure of $K(x)$ in $L_2\,$. Since $L_2P=K$,
$P$ must be non-trivial on $L_0\,$. We have $L_0P=K$ and by
Lemma~\ref{trac}, $(L_0,v_P)$ is a tame field and $v_PL_0$ is divisible.
Consequently, $v_PL_0$ is existentially closed in $v_PL_2\,$. Trivially,
$L_0P=K$ is existentially closed in $L_2P=K$. So we can employ
Theorem~\ref{tameAKE} to deduce that $L_0$ is existentially closed in
$L_2$.

Now it suffices to prove that $K$ is existentially closed in $L_0$
because then by transitivity, $K$ is existentially closed in $L_2$
and thus also in its subfield $L$. We only have to show that $K$ is
existentially closed in every subfield $F$ of $L_0$ which is finitely
generated over $K$. Since $L_0$ is an algebraic extension of $K(x)$
and $P$ is trivial on $K$, it follows that $P$ is discrete on $F$.
Since $K$ is large, we can now infer from Theorem~\ref{lfecKt} that $K$
is existentially closed in $F$. This completes our proof.      \QED

%\newpage\noindent
%
%ÄÄÄÄÄÄÄÄÄÄÄÄÄÄÄÄÄÄÄÄÄÄÄÄÄÄÄÄÄÄÄÄÄÄÄÄÄÄÄÄÄÄÄÄÄÄÄÄÄÄÄÄÄÄÄÄÄÄÄÄÄÄÄÄÄÄÄÄÄÄÄ
%
\section{Appendix}
Let ${\cal L}_0$ be a first order language, and ${\cal L}$ an extension
of ${\cal L}_0$ by new relation symbols. Take $M_0$ to be any
${\cal L}_0$-structure. Further, let ${\cal T}$ be a set of universal
${\cal L}$-sentences, and ${\cal A}$ a set of existential
${\cal L}$-sentences.

Now let $X$ be the set of ${\cal L}$-expansions $M$ of $M_0\,$ such
that $M\models {\cal T}$. On $X$, we consider the topology
${\cal X}_{\cal A}$ whose basic open sets are the sets of the form
\[\{M\in X\mid M\models {\cal A}'\}\>,\;\;\; {\cal A}'
\mbox{ a finite subset of } {\cal A}\;.\]
Replacing ${\cal A}$ by its closure under finite conjunctions if
necessary, we may assume that all basic open sets are of the form
\[X_{\varphi}:=\{M\in X\mid M\models \varphi\}\;,\]
where $\varphi\in {\cal A}$.

\begin{theorem}                             \label{qco}
$(X,{\cal X}_{\cal A})$ is quasi-compact.
\end{theorem}
\begin{proof}
Take a collection $\{X_{\varphi_i}\mid i\in I\}$ of
basic open sets, with $\varphi_i\in {\cal A}$. Assume that
\[X\;=\;\bigcup_{i\in J} X_{\varphi_i}\]
does not hold for any finite subset $J$ of $I$. Then we have to show
that it also does not hold for $J=I$.

By our assumption, for every finite subset $J$ of $I$ there is an
expansion $M_J\in X$ which is not in $\bigcup_{i\in J} X_{\varphi_i}\,$.
That is, $M_J\models\bigwedge_{i\in J}
\neg\varphi_i\,$. We take ${\cal T}_0$ to be the elementary
${\cal L}_0(M_0)$-theory of $M_0\,$, where ${\cal L}_0(M_0)$ is the
language obtained from ${\cal L}_0$ by adjoining a constant symbol for
every element of $M_0\,$. We see that for every finite subset $J$ of
$I$, the theory ${\cal T}_0\cup {\cal T}\cup \{\neg\varphi_i\mid i\in
J\}$ has a model, namely, $M_J\,$. By the semantical compactness theorem
of first order logic, we conclude that also the theory ${\cal T}_0\cup
{\cal T}\cup \{\neg\varphi_i\mid i\in I\}$ has a model $M^*$.

Since $M^*\models {\cal T}_0\,$, we know that $M_0$ is an elementary
substructure of the ${\cal L}_0$-reduct of $M^*$. Let us denote by
$M'$ the ${\cal L}$-structure which we obtain by restricting the new
relations of $M^*$ to the universe of $M_0\,$. Then $M_0$ is the
${\cal L}_0$-reduct of $M'$, that is, $M'$ is an ${\cal L}$-expansion
of $M_0\,$. Since ${\cal T}$ consists of universal sentences, we also
have that $M'\models {\cal T}$. Hence, $M'\in X$. But for all $i\in
I$, $M^*\models \neg\varphi_i$ and since $\neg\varphi_i$ is a universal
sentence, we also have that $M'\models \neg\varphi_i\,$.
Therefore,
\[M'\notin \bigcup_{i\in I} X_{\varphi_i}\;.\]
\end{proof}

Replacing ${\cal T}$ by ${\cal T}\cup\{\varphi\}$ for any quantifier free
sentence $\varphi\in {\cal A}$, we obtain:
\begin{corollary}
If $\varphi\in {\cal A}$ is quantifier free, then $X_{\varphi}$ is
quasi-compact. Hence if ${\cal A}$ consists only of quantifier free
sentences, then every basic open set is quasi-compact.
\end{corollary}

Now we consider the following conditions on ${\cal A}$:
\sn
(T$_0$) \quad \ for all $M_1,M_2\in X$, $M_1\ne M_2\,$, there is some
$\varphi\in {\cal A}$ such that\par
\qquad $M_1\notin X_{\varphi}\ni M_2$ or $M_2\notin X_{\varphi}\ni
M_1\,$.
\sn
(NEG) \ ${\cal A}$ is closed under negation.
\sn
Note: if ${\cal A}$ satisfies (NEG), then it consists solely of
quantifier free sentences.

\begin{proposition}                               \label{tdcH}
If ${\cal A}$ satisfies (T$_0$), then $(X,{\cal X}_{\cal A})$ is a
T$_0$-space. If in addition, ${\cal A}$ satisfies (NEG), then $(X,
{\cal X}_{\cal A})$ is a totally disconnected compact Hausdorff space
and its basic open sets are both open and closed.
\end{proposition}
\begin{proof}
The first assertion is obvious. Now assume that both (T$_0$) and (NEG)
hold. Since ${\cal A}$ is closed under negation, the complement
$X_{\neg\varphi}$ of a basic open set $X_{\varphi}$ is also basic open.
Therefore, each set $X_{\varphi}$ is open and closed, and
$(X,{\cal X}_{\cal A})$ is totally disconnected. If $M_1,M_2\in X$
such that $M_1\ne M_2\,$, then by (T$_0$), there is some $\varphi\in
{\cal A}$ such that $M_1\in X_{\varphi}$ and $M_2\in X_{\neg\varphi}$.
This shows that $(X,{\cal X}_{\cal A})$ is Hausdorff. The compactness
follows from Theorem~\ref{qco}.
\end{proof}

Let us observe the following fact, which we will not need any further:
\begin{corollary}
If ${\cal A}$ satisfies (T$_0$) and (NEG), then ${\cal X}_{\cal A}=
{\cal X}_{\cal Q}$ where $\cal Q$ denotes the set of all quantifier free
elementary ${\cal L}$-sentences.
\end{corollary}
\begin{proof}
It is clear that $Q$ satisfies (NEG). Since ${\cal A}\subseteq
{\cal Q}$, we know that $\cal Q$ also satisfies (T$_0$). Hence,
$(X,{\cal X}_{\cal Q})$ is a totally disconnected compact Hausdorff
space. Since the same holds for $(X,{\cal X}_{\cal A})$ and since
${\cal X}_{\cal A}\subseteq {\cal X}_{\cal Q}\,$, we must have that
${\cal X}_{\cal A}={\cal X}_{\cal Q}\,$.
\end{proof}

\begin{theorem}                             \label{Xspsp}
Suppose that ${\cal A}$ satisfies (NEG), and that ${\cal B}$ is a subset
of ${\cal A}$ which satisfies (T$_0$). Then $(X,{\cal X}_{\cal B})$ is a
spectral space and ${\cal X}_{\cal A}$ is its associated patch topology.
\end{theorem}
\begin{proof}
If ${\cal A}$ contains a subset ${\cal B}$ which satisfies (T$_0$), then
also ${\cal A}$ satisfies (T$_0$). Thus, our assertion follows from
Proposition~\ref{tdcH} together with Proposition 7 of [H].
\end{proof}

\pars
In order to apply this result to spaces of valuation rings on a fixed
field $F$, we take ${\cal L}_0$ to be the language of rings together
with a set ${\cal F}$ of constant symbols for all elements of $F$, i.e.,
${\cal L}_0= \{+,-, \cdot,0,1\}\cup {\cal F}$. Further,
we set ${\cal L}= {\cal L}_0\cup\{{\cal O}\}$, where ${\cal O}$ is a
unary predicate symbol. We take ${\cal T}_v$ to consist of the
following universal ${\cal L}$-sentences which say that (the
interpretation of) ${\cal O}$ is a subring which is a valuation ring:

1) \ $\forall x\forall y:\; ({\cal O}(x)\wedge {\cal O}(y))
\>\rightarrow\> ({\cal O}(x-y)\wedge {\cal O}(xy))$,

2) \ $\forall x\forall y:\; xy=1 \>\rightarrow\>
({\cal O}(x)\vee {\cal O}(y))$.

\sn
Further, we let ${\cal T}_S$ be an arbitrary set of universal
${\cal L}$-sentences; they single out a universally definable
subset $S$ of valuation rings on $F$. We set $M_0=F$, ${\cal T}=
{\cal T}_v\cup {\cal T}_S\,$,
\[{\cal A}\;=\;\{{\cal O}(a)\,,\, \neg {\cal O}(a)\mid a\in {\cal F}\}
\quad \mbox{\ \ and\ \ } \quad
{\cal B}\;=\;\{{\cal O}(a)\mid a\in {\cal F}\}\,.\]
In this setting, every expansion $M\in X$ is given by the choice of a
valuation ring ${\cal O}$ of $F$ which satisfies ${\cal T}_S\,$. Hence
we have a bijection between $X$ and the set $S(F\,;\,{\cal T}_S)$ of all
valuation rings on $F$ which satisfy ${\cal T}_S\,$, and we identify
these sets. Then ${\cal X}_{\cal B}$ is the Zariski topology on $S(F\,;\,
{\cal T}_S)$. We note that for $a\ne 0$, $a\notin {\cal O}$ is
equivalent to $a^{-1} \in {\cal M}$. Therefore, ${\cal X}_{\cal A}$ is
the patch topology on $S(F\,;\,{\cal T}_S)$, as defined by the basic
open sets (\ref{PT}) in the introduction.

Suppose that ${\cal O}_1$ and ${\cal O}_2$ are two distinct valuation
rings on $F$. Then there is some $a\in F$ such that $a\in {\cal O}_1
\setminus {\cal O}_2$ or $a\in {\cal O}_2 \setminus {\cal O}_1\,$. So we
see that ${\cal B}$ satisfies (T$_0$). Clearly, ${\cal A}$ satisfies
(NEG). Hence by Theorem~\ref{Xspsp}, we obtain:

\begin{theorem}                             \label{Zssps}
The Zariski space $S(F\,;\,{\cal T}_S)$ together with the Zariski
topology given by the basic open sets (\ref{ZT}) is a spectral space,
and the topology given by the basic open sets (\ref{PT}) is its
associated patch topology.
\end{theorem}
\n
Taking ${\cal T}_S=\emptyset$, we see that these assertions hold in
particular for the Zariski space of all valuation rings on a fixed
field $F$.

\pars
As we have constant symbols for all elements of $K$ in our language
${\cal L}$, we can take ${\cal T}_S$ to be a set of quantifier free
${\cal L}$-sentences in ${\cal T}$ which state that ${\cal O}\cap K$ is
the valuation ring ${\cal O}_{\Pzero}$ of a given place $\Pzero$ on $K$.
Then $S(F\,;\,{\cal T}_S) = S(F|K\,;\,\Pzero)$.

\begin{corollary}                            \label{Zsspsc}
The assertions of Theorem~\ref{Zssps} hold in particular for
$S(F|K\,;\,\Pzero)$ and for $S(F|K)$.
\end{corollary}

\parb
Let us demonstrate the usefulness of Theorem~\ref{Xspsp} by two more
applications, the first of which can be seen as a generalization of
Theorem~\ref{Zssps}. In both applications, let $R$ be a commutative ring
with unity.

The set Spv$(R)$ of all valuations (in the sense of [HU--KN]) on $R$ is
called the \bfind{valuation spectrum} of $R$. As $R$ will in general not
be a field, instead of a unary predicate for the valuation ring we have
to use a binary relation for what is called \bfind{valuation
divisibility}. We write $x|y\Longleftrightarrow vx\leq vy$. In order to
encode a valuation, the relation $|$ has to satisfy the following
universal axioms (cf.\ [HU--KN]):

1) \ $\forall x \forall y :\; x|y\>\vee\>y|x$,

2) \ $\forall x \forall y \forall z:\; (x|y\>\wedge\>y|z)\>\rightarrow\>
x|z$,

3) \ $\forall x \forall y \forall z:\; (x|y\>\wedge\>x|z)\>\rightarrow\>
x|y+z$,

4) \ $\forall x \forall y \forall z:\; x|y\>\rightarrow\> xz|yz$,

5) \ $\forall x \forall y \forall z:\; (xz|yz\>\wedge\>
0\!\not|\;z)\>\rightarrow\> x|y$,

6) \ $0\!\not|\; 1$.

\sn
We take ${\cal R}$ to be a set of constant symbols for all elements of
$R$, and ${\cal L}_0= \{+,-,\cdot,0,1\}\cup {\cal R}$. Further, we take
${\cal L}= {\cal L}_0\cup\{\>|\>\}$ with $|$ a binary predicate symbol,
$M_0=R$, and ${\cal T}$ to consist of the above axioms. One
possible pair of topologies on the valuation spectrum is given by the
sets
\[{\cal A}\;=\;\{a|b\,,\, a\!\not|\;b\mid a,b\in {\cal R}\} \quad
\mbox{\ \ and\ \ } \quad {\cal B}\;=\;\{a|b\mid a,b\in {\cal R}\}\,.\]
It is clear that ${\cal B}$ satisfies (T$_0$). Now Theorem~\ref{Xspsp}
shows:
\begin{theorem}
The valuation spectrum of $R$ is a spectral space, for the topology
whose basic open sets are of the form $\{v\in\mbox{\rm Spv}(R)\mid va_1
\leq vb_1\,,\ldots,\, va_k\leq vb_k\}$. The basic open sets of its
patch topology are of the form $\{v\in\mbox{\rm Spv}(R)\mid va_1\leq
vb_1\,,\ldots,\, va_k\leq vb_k\>;$ $vc_1<vd_1\,,\ldots,\, vc_\ell<
vd_\ell\}$.
\end{theorem}

Note that Huber and Knebusch prefer to work with different topologies
which essentially are obtained from the above by adding the condition
$a\not= 0$. Yet the above theorem remains true (cf.\ [HU--KN]).

\parm
The \bfind{real spectrum} of the ring $R$ is the set of preorderings $P$
with support a prime ideal and satisfying $P\cup -P=R$ and $-1\notin P$
(cf.\ [B--C--R], [C--R], [KN--S]). That is, it can be presented by all
(interpretations of) unary predicates $P$ which satisfy the following
universal axioms:

1) \ $\forall x\forall y:\; (P(x)\wedge P(y))
\>\rightarrow\> (P(x+y)\wedge P(xy))$,

2) \ $\forall x:\; P(x^2)$,

3) \ $\forall x:\; P(x)\vee P(-x)$,

4) \ $\neg P(-1)$,

5) \ $\forall x\forall y:\; (P(xy)\wedge P(-xy))
\>\rightarrow\> (P(x)\wedge P(-x))\vee(P(y)\wedge P(-y))$.

\sn
We take ${\cal L}_0$ and $M_0$ as before. Further, we take
${\cal L}= {\cal L}_0\cup\{P\}$ with $P$ a unary predicate symbol, and
${\cal T}$ to consist of the above axioms. The topologies on the real
spectrum are given by the sets
\[{\cal A}\;=\;\{P(a)\,,\, \neg P(a)\mid a\in {\cal R}\} \quad
\mbox{\ \ and\ \ } \quad {\cal B}\;=\;\{P(a)\mid a\in {\cal R}\}\,.\]
Again, ${\cal B}$ satisfies (T$_0$). So Theorem~\ref{Xspsp} shows:
\begin{theorem}
The real spectrum of $R$ is a spectral space, for the topology whose
basic open sets are of the form $\{P\in\mbox{\rm Sper}(R)\mid a_1
,\ldots,a_k\in P\}$. The basic open sets of its patch topology are of
the form $\{P\in\mbox{\rm Sper}(R)\mid a_1 ,\ldots,a_k\in P\>;\>
b_1,\ldots,b_\ell\notin P\}$.
\end{theorem}

\mn
A comparable approach to spectral spaces using model theory has been
worked out by R.~Berr in [BE].

\mn
%\bn
%\bn
%\newpage\noindent
%\end{document}
\bn
{\bf References}
\newenvironment{reference}%
{\begin{list}{}{\setlength{\labelwidth}{5em}\setlength{\labelsep}{0em}%
\setlength{\leftmargin}{5em}\setlength{\itemsep}{-1pt}%
\setlength{\baselineskip}{3pt}}}%
{\end{list}}
\newcommand{\lit}[1]{\item[{#1}\hfill]}
\begin{reference}
\lit{[A]} {Abhyankar, S.$\,$: {\it On the valuations centered in a local
domain}, Amer.\ J.\ Math.\ {\bf 78} (1956), 321--348}
\lit{[B--C--R]} {Bochnak, J.\ -- Coste, M.\ -- Roy, M.-F.$\,$: {\it
G\'eom\'etrie alg\'ebrique r\'eelle}, Springer, Berlin (1987);
English translation: {\it Real Algebraic Geometry}, Ergebnisse der
Mathematik und ihrer Grenzgebiete, Vol.\ {\bf 36}, Springer, Berlin
(1998)}
\lit{[BE]} {Berr, R.$\,$: {\it Spectral spaces and first order
theories}, manuscript}
\lit{[BO]} {Bourbaki, N.$\,$: {\it Commutative algebra}, Paris (1972)}
\lit{[C--R]} {Coste, M.\ -- Roy, M.-F.$\,$: {\it La topologie du
spectre r\'eel}, in: Ordered Fields and Real Algebraic Geometry
(D.\ W.\ Dubois and T.\ Recio, eds.), Contemporary Mathematics, Vol.\
{\bf 8}, Amer.\ Math.\ Soc., Providence (1982), 27-59}
\lit{[EN]} {Endler, O.$\,$: {\it Valuation theory}, Berlin (1972)}
\lit{[ER]} {Ershov, Yu.\ L.$\,$: {\it Rational points over Hensel
fields} (in Russian), Algebra i Logika {\bf 6} (1967), 39--49}
%\lit{[ER]} {Ershov, Yu.$\,$: {\it }}
\lit{[HA]} {Hartshorne, R.$\,$: {\it Algebraic geometry}, Springer
Graduate Texts in Math.\ {\bf 52}, New York (1977)}
\lit{[HO]} {Hochster, M.$\,$: {\it Prime ideal structure in commutative
rings}, Trans.\ Amer.\ Math.\ Soc.\ {\bf 142} (1969), 43--60}
\lit{[HU--KN]} {Huber, M.\ -- Knebusch, M$\,$: {\it On Valuation
Spectra}, in: Recent Advances in Real Algebraic Geometry and
Quadratic Forms (W.~B.~Jacob, T.-Y.~Lam and R.~O.~Robson, eds.),
Contemporary Mathematics {\bf 155}, Amer.\ Math.\ Soc.\ (1994),
167--206}
\lit{[J--R]} {Jarden, M.\ -- Roquette, P.$\,$: {\it The Nullstellensatz
over $\wp$--adically closed fields}, J.\ Math.\ Soc.\ Japan {\bf 32}
(1980), 425--460}
\lit{[KA]} {Kaplansky, I.$\,$: {\it Maximal fields with valuations I},
Duke Math.\ Journ.\ {\bf 9} (1942), 303--321}
\lit{[K--K]} {Knaf, H.\ -- Kuhlmann, F.-V.$\,$: {\it Abhyankar places
admit local uniformization in any characteristic}, preprint}
\lit{[KN--S]} {Knebusch, M.\ -- Scheiderer, C.$\,$: {\it Einf\"uhrung in
die reelle Algebra}, Vieweg, Braunschweig (1989)}
\lit{[K1]} {Kuhlmann, F.-V.: {\it Henselian function fields and tame
fields}, extended version of Ph.D.\ thesis, Heidelberg
(1990)}
\lit{[K2]} {Kuhlmann, F.-V.: {\it Valuation theory of fields, abelian
groups and modules}, preprint, Heidelberg (1996), to appear in the
``Algebra, Logic and Applications'' series (formerly Gordon and Breach,
eds.\ A.~Mac\-intyre and R.~G\"obel). Preliminary versions of several
chapters are available on the web site\\
http://math.usask.ca/$\,\tilde{ }\,$fvk/Fvkbook.htm.}
\lit{[K3]} {Kuhlmann, F.-V.: {\it On local uniformization in arbitrary
characteristic}, The Fields Institute Preprint Series, Toronto (1997)}
\lit{[K4]} {Kuhlmann, F.--V.$\,$: {\it Valuation theoretic and model
theoretic aspects of local uniformization},
in: Resolution of Singularities --- A
Research Textbook in Tribute to Oscar Zariski. Herwig Hauser, Joseph
Lipman, Frans Oort, Adolfo Quiros (eds.), Progress in Mathematics Vol.\
{\bf 181}, Birkh\"auser Verlag Basel (2000), 381-456}
\lit{[K5]} {Kuhlmann, F.-V.: {\it Value groups, residue fields and bad
places of rational function fields}, to appear in Trans.\ Amer.\ Math.\
Soc.}
\lit{[K6]} {Kuhlmann, F.--V.$\,$: {\it Every place admits local
uniformization in a finite purely wild extension of the function field},
preprint}
\lit{[K7]} {Kuhlmann, F.--V.$\,$: {\it The model theory of tame valued
fields}, preprint}
\lit{[K8]} {Kuhlmann, F.--V.$\,$: {\it Algebraic independence of
elements in completions and maximal immediate extensions of valued
fields}, preprint}
\lit{[K--P]} {Kuhlmann, F.--V.\ -- Prestel, A.$\,$: {\it On places of
algebraic function fields}, J.\ reine angew.\ Math.\ {\bf 353}
(1984), 181--195}
\lit{[L]} {Lang, S.$\,$: {\it Some applications of the local
uniformization theorem},  Amer.\ J.\ Math.\ {\bf 76}, (1954), 362--374}
\lit{[POP1]} {Pop, F.$\,$: {\it Embedding problems over large fields},
Forschungsschwerpunkt Arithmetik, Universit\"at Heidelberg /
Universit\"at Mannheim, Heft Nr.\ {\bf 15} (1995)}
\lit{[POP2]} {Pop, F.$\,$: {\it Embedding problems over large fields},
Annals of Math.\ {\bf 144} (1996), 1--34}
\lit{[P--Z]} {Prestel, A.\ -- Ziegler, M.$\,$: {\it Model theoretic
methods in the theory of topological fields}, J.\ reine angew.\ Math.\
{\bf 299/300} (1978), 318--341}
\lit{[RI]} {Ribenboim, P.$\,$: {\it Th\'eorie des valuations}, Les
Presses de l'Uni\-versit\'e de Mont\-r\'eal, Mont\-r\'eal, 2nd ed.\
(1968)}
\lit{[RO]} {Robinson, A.$\,$: {\it Complete Theories}, Amsterdam
(1956)}
%\lit{[S--M]} {Schwartz, N.\ -- Madden, J.$\,$: {\it Semi-algebraic
%function rings and reflectors of partially ordered rings}, Lecture Notes
%in Math.\ {\bf 1712}, Berlin (1999)}
\lit{[V]} {Vaqui\'e, M.$\,$: {\it Valuations},
in: Resolution of Singularities --- A
Research Textbook in Tribute to Oscar Zariski. Herwig Hauser, Joseph
Lipman, Frans Oort, Adolfo Quiros (eds.), Progress in Mathematics Vol.\
{\bf 181}, Birkh\"auser Verlag Basel (2000), 381-456}
\lit{[W]} {Warner, S.$\,$: {\it Topological fields}, Mathematics
studies {\bf 157}, North Holland, Amsterdam (1989)}
\lit{[Z1]} {Zariski, O.$\,$: {\it Local uniformization on
algebraic varieties}, Ann.\ Math.\ {\bf 41} (1940), 852--896}
\lit{[Z2]} {Zariski, O.$\,$: {\it The compactness of the Riemann
manifold of an abstract field of algebraic functions}, Bull.\ Amer.\
Math.\ Soc.\ {\bf 50} (1944), 683--691}
\lit{[Z--SA]} {Zariski, O.\ -- Samuel, P.$\,$: {\it Commutative
Algebra}, Vol.\ II, New York--Heidel\-berg--Berlin (1960)}
\end{reference}
\adresse
\end{document}